\documentclass[11pt,a4paper]{article}
\usepackage{amsmath,amsfonts,amssymb,amsthm,enumerate,graphicx,mathtools,tikz,ifthen,titling,caption}
\usepackage{hyperref}
\usepackage[margin=2.5cm]{geometry}
\usepackage{xypic} %for diagrams
\usepackage{wasysym}
\usepackage{mathrsfs} %for script fonts
\usepackage{xspace}
\usepackage{fixltx2e}

\usepackage{authblk}
\usepackage{tocloft}
\usepackage[normalem]{ulem}

\usepackage{pdflscape}
\usepackage[toc,page,header]{appendix}
\usepackage{listings}

\usepackage{color}

\usetikzlibrary{hobby, intersections, math}

\newtheorem{thm}{Theorem}[section]
\newtheorem{lem}[thm]{Lemma}
\newtheorem{prop}[thm]{Proposition}
\newtheorem{cor}[thm]{Corollary}
\newtheorem{hyp}[thm]{Hypothesis}

\theoremstyle{definition}
\newtheorem{defn}[thm]{Definition}
\newtheorem{ex}[thm]{Example}

\newtheorem{rem}[thm]{Remark}

\newcommand{\ol}{\overline}

\newcommand{\defbold}{\textbf}
\newcommand{\Min}{\mathrm{Min}}

\newcommand{\inv}{^{-1}}

\newcommand{\N}{\mathrm{N}}

\newcommand{\tdlc}{t.d.l.c.\@\xspace}
\newcommand{\Fix}{\mathrm{Fix}}

\newcommand{\hb}{\mathrm{hyp}}

\newcommand{\para}{\mathrm{par}}
\newcommand{\con}{\mathrm{con}}
\newcommand{\nub}{\mathrm{nub}}

\newcommand{\triv}{\{1\}}

\newcommand{\TA}{(\mathrm{TA})}
\newcommand{\TB}{(\mathrm{TB})}
\newcommand{\GTA}{(\mathrm{GTA})}
\newcommand{\GTp}{(\mathrm{GT}_+)}
\newcommand{\GTm}{(\mathrm{GT}_-)}

\newcommand{\Aut}{\mathrm{Aut}}
\newcommand{\Isom}{\mathrm{Isom}}

\newcommand{\Sym}{\mathrm{Sym}}

\newcommand{\Nb}{\mathbb{N}}

\newcommand{\Qb}{\mathbb{Q}}
\newcommand{\Rb}{\mathbb{R}}

\newcommand{\Zb}{\mathbb{Z}}

\newcommand{\mc}[1]{\mathcal{#1}}

\newcommand{\grp}[1]{\langle #1 \rangle}

\begin{document}

\title{The scale function for locally compact groups acting on non-positively curved spaces}

\preauthor{\large}
\DeclareRobustCommand{\authoring}{
\renewcommand{\thefootnote}{\arabic{footnote}}
\begin{center}Colin D. Reid\textsuperscript{1}\footnotetext[1]{Research supported in part by ARC grant FL170100032.}\\
\href{mailto:colin@reidit.net}{colin@reidit.net}
%\textsuperscript{1}The University of Newcastle, School of Information and Physical Sciences, Callaghan, NSW 2308, Australia
\end{center}
}
\author{\authoring}
\postauthor{\par}

\maketitle

\begin{abstract}
Let $G$ be a totally disconnected, locally compact (\tdlc) group.  The scale $s_G(g)$ of $g \in G$ in the sense of Willis is given by the minimum value of the index $|gUg\inv:U \cap gUg\inv|$ as $U$ ranges over the compact open subgroups; the theory associated to the scale has been very successful in describing general dynamical features of automorphisms of \tdlc groups.  

We focus on the case where $G$ acts properly and continuously by isometries on a geodesic space $X$, where $X$ is complete CAT(0) or proper and Gromov-hyperbolic, and $g \in G$ is hyperbolic.  In this context, we find geometric descriptions of the parabolic and contraction groups, tidy subgroups, and structures in the $G$-action that encode the scale, including criteria for $g$ to have scale $1$.
\end{abstract}

%\tableofcontents

\section{Introduction}

\subsection{Background and scope}

Scale theory is an approach to studying the dynamics of automorphisms of totally disconnected, locally compact (\tdlc) groups.  Developed principally by G. Willis and his collaborators, starting with the 1994 article \cite{Willis94}, it is a major plank of the general theory of \tdlc groups as it exists today.  Much of its power lies in the fact that its main results apply to all \tdlc groups, without any assumptions about linearity, finiteness properties or geometric structure.  However, an unavoidable consequence of this generality is that the core concepts are quite abstract in nature, concerned purely with the structure of the \tdlc group $G$ as a topological group.

The goal of this article is to give `geometric' interpretations to some basic concepts from scale theory in a setting which, while not totally general, covers many classes of \tdlc groups of interest in practice.  For concreteness and to take advantage of well-established results on the geometric side, we only consider inner automorphisms (by itself this is not much of a restriction, since $G \rtimes \grp{\alpha}$ is a \tdlc group for any \tdlc group $G$ and automorphism $\alpha$), and we will assume the \tdlc group $G$ is acting properly and continuously on a geodesic space $X$ with at least one of the following properties:
\begin{enumerate}[(a)]
\item $X$ is a complete metric space and satisfies the CAT(0) inequality;
\item $X$ is proper (that is, closed balls are compact) and $\delta$-hyperbolic in the sense of Gromov.
\end{enumerate}
We will refer to such spaces as \defbold{NPC spaces}.  The given hypotheses (a) and (b) can be considered `test cases' for an approach that will hopefully extend to a more general setting of non-positively curved actions, and possibly beyond.  Both types of NPC space have a natural notion of visual boundary, which we denote $\partial X$.

The foundational concepts of scale theory are the scale function and tidy subgroups.  Given a \tdlc group $G$ and $g \in G$, the \defbold{scale} $s_G(g)$ is the minimum value of $|gUg\inv: U \cap gUg\inv|$ as $U$ ranges over the compact open subgroups of $G$, and $U$ is said to be \defbold{minimizing} if it achieves this minimum value.  Willis characterized the structure of the minimizing compact open subgroups (\cite[Theorem~3.1]{WillisFurther}, see Theorem~\ref{thm:Willis}).

To understand a given element $g \in G$ in terms of the scale, it is important to consider the pair $(s_G(g),s_G(g\inv))$.  By definition, the scale takes positive integer values; from a dynamical perspective, the key distinction is between a value equal to $1$, and a value greater than $1$.  In particular, $g$ is \defbold{uniscalar} (in $G$) if $s_G(g) = s_G(g\inv) = 1$.  For example, if $G$ has a compact open normal subgroup, then clearly $G$ is uniscalar.  However, the structure of more general uniscalar \tdlc groups is mysterious at present, even in what should be the most favourable situation geometrically, namely when $G$ acts properly and cocompactly on a locally finite tree.  A major goal of this article is to understand more about elements of scale $1$.

In terms of the action on the NPC space, the element $g \in G$ could be one of three types of isometry: bounded, parabolic or hyperbolic.  (Although conceptually similar, the conventional definitions are different for $\delta$-hyperbolic versus CAT(0) spaces; see Section~\ref{sec:types} for details.)  If the forward orbit $(g^nx)_{n \ge 0}$ (for $x \in X$) travels towards some point $\xi_+$ at infinity, then $\xi_+$ does not depend on $x$, and we call it the \defbold{attracting point} $\xi_+(g)$ of $g$; similarly we define the \defbold{repelling point} $\xi_-(g) := \xi_+(g^{-1})$.  The \defbold{translation length} of $g$ is given by $|g| := \lim_{n \rightarrow +\infty}d(x,g^nx)/n$ for any point $x \in X$; in particular, $|g|=0$ if $g$ is bounded, while $|g|>0$ if $g$ is hyperbolic.

If $g$ is bounded, then it is uniscalar for more or less trivial reasons, so we are not interested in this case.  For many well-behaved actions of \tdlc groups, parabolic isometries are ruled out by general results (see for example \cite{BridsonPolyhedral} and \cite[Theorem 20]{BMW}), and general tools for studying parabolic isometries are somewhat limited, so we largely avoid this case.  This leaves the main focus on when $g$ acts as a hyperbolic isometry, which ensures the attracting and repelling points of $g$ exist and are distinct.

In this article we focus on scale theory for a cyclic subgroup $\grp{g}$; we note however that there is a higher rank generalization of the scale theory of cyclic subgroups, known as the theory of flat subgroups.  The study of geometric structure in NPC spaces associated to flat subgroups will be postponed to a subsequent article.

\subsection{The stabilizer of the attracting point}

Given $g \in G$, the \defbold{parabolic subgroup} $\para_G(g)$ is the set of $x \in G$ such that the set $\{g^nxg^{-n} \mid n \ge 0\}$ has compact closure; some important connections with the scale were shown by Baumgartner--Willis \cite{BaumgartnerWillis}.  In our present context, when $g$ is unbounded we relate $\para_G(g\inv)$ to the stabilizer of the attracting point of $g$.

\begin{thm}[See Section~\ref{sec:parabolic}]\label{thm:para_stable}
Let $X$ be an NPC space, let $G$ be a locally compact group acting continuously by isometries and let $\xi \in \partial X$.  Let $g \in G$ have unbounded action on $X$ with an attracting point $\xi$.  Then $\para_G(g\inv) \le G_{\xi}$; if $g$ is hyperbolic and the action of $G$ proper, then $\para_G(g\inv) = G_{\xi}$.
\end{thm}

Theorem~\ref{thm:para_stable}, together with general scale theory results, provide equivalent ways to obtain the scale in terms of $G_\xi$.

\begin{cor}[{See Corollary~\ref{cor:parabolic_scale}}]
Let $X$ be an NPC space and let $G$ be a \tdlc group acting continuously by isometries. Suppose $g \in G$ has unbounded action, with an attracting point $\xi$.
\begin{enumerate}[(i)]
\item We have $s_G(g) = s_{G_{\xi}}(g)$.
\item If $G$ acts properly on $X$ and $g$ is hyperbolic, then $s_G(g) = \Delta_{G_{\xi}}(g)$.
\end{enumerate}
\end{cor}

\subsection{Hyperbolic isometries of scale 1}

Our next main result is to give several equivalent criteria for when a hyperbolic element has scale $1$.

\begin{defn}
Let $X$ be an NPC space and let $G$ be a group acting on $X$ by isometries.  We write $\partial^{\hb}_X G$ for the set of points on the boundary of $X$ that are limit points of hyperbolic elements of $G$.  Given $\xi_1,\xi_2 \in \partial^{\hb}_X G$, we write $\xi_1 \rightarrow \xi_2$ if whenever $g \in G$ is hyperbolic with attracting point $\xi_1$, then $g$ has repelling point $\xi_2$.  If $\xi_1 \rightarrow \xi_2$ and $\xi_2 \rightarrow \xi_1$, we write $\xi_1 \leftrightarrow \xi_2$.
\end{defn}

\begin{thm}[See Section~\ref{sec:uniscalar_hyperbolic}]\label{thm:uniscalar_hyperbolic}
Let $X$ be an NPC space, let $G$ be a \tdlc group acting properly and continuously by isometries and let $g \in G$ be hyperbolic, with attracting point $\xi_+$ and repelling point $\xi_-$.  Then the following are equivalent:
\begin{enumerate}[(i)]
\item $s_G(g)=1$;
\item $s_{G_{\xi_+}}(g)=1$;
\item $\Delta_{G_{\xi_+}}(g) = 1$;
\item $G_{\xi_-}$ is open;
\item $G_{\xi_-,\xi_+}$ is open in $G_{\xi_+}$;
\item $G_{\xi_+} \le G_{\xi_-}$;
\item $\xi_+ \rightarrow \xi_-$.
\end{enumerate}
\end{thm}

\begin{cor}
Let $X$ be an NPC space, let $G$ be a \tdlc group acting properly and continuously by isometries, let $\xi \in \partial^{\hb}_X G$ and let $S^+_{\xi}$ be the set of hyperbolic elements of $G$ with attracting point $\xi$.  Then either $s_G(g)=1$ for all $g \in S^+_{\xi}$ (if $\xi \rightarrow \xi'$ for some $\xi' \in \partial^{\hb}_X G$) or $s_G(g)>1$ for all $g \in S^+_{\xi}$ (otherwise).  Moreover, either $s_G(g\inv)=1$ for all $g \in S^+_{\xi}$ (if $G_{\xi}$ is open) or $s_G(g\inv)>1$ for all $g \in S^+_{\xi}$ (if $G_{\xi}$ is not open).
\end{cor}

\begin{rem}
When $G$ is a hyperbolic \tdlc group and $X$ is one of its Cayley--Abels graphs, the equivalence of (i) and (iv) was effectively proved by \cite[Lemma 21]{BMW}.

The relation $\rightarrow$ has the following structure on $\partial^{\hb}_X G$:
\begin{itemize}
\item Say $\xi \in \partial^{\hb}_X G$ is \defbold{uniscalar} for $G$ if there is some $g \in G$ hyperbolic, with $\xi_+(g) = \xi$, such that $s_G(g) = s_G(g\inv)=1$.  Then the set $\partial^{\hb,u}_X G$ of uniscalar limit points decomposes into disjoint pairs $\xi \leftrightarrow \xi^*$, with $G_{\xi} = G_{\xi^*}$ being open in $G$ and no other arrows incident with $\partial^{\hb,u}_X G$.

Thus there is an involution on $\partial^{\hb,u}_X G$ that commutes with the natural action of $G$, which can be thought of as an `inverse at infinity' for the uniscalar hyperbolic elements of $G$.  For example, if $G$ is a finitely generated hyperbolic group acting by left translation on one of its Cayley graphs $X$ (with vertices identified with $G$) and $g \in G$ has infinite order (and hence $g$ is hyperbolic on $X$, see \cite[Chapitre 9.3]{CDP}), then $\xi = \xi_+(g)$ is represented as $\lim_{n \rightarrow +\infty} g^n$, and the `inverse' of $\xi$ is $\xi_-(g) = \lim_{n \rightarrow +\infty} g^{-n}$.
\item Let $\xi \in \partial^{\hb}_X G$ and write $S^+_\xi$ for the set of hyperbolic $g \in G$ with attracting point $\xi$.  Suppose there is $g \in S^+_\xi$ such that $s_G(g)>1$ but $s_G(g\inv)=1$, and let $R_\xi$ be the set of repelling points of $h \in S^+_\xi$.  Then $G_\xi$ is open and its orbits on $R_\xi$ are uncountable.  The arrows incident with $\xi$ consist of an arrow from each element of $R_\xi$ to $\xi$.
\item Suppose there is a hyperbolic $g \in G$ with attracting point $\xi$, such that $s_G(g)>1$ and $s_G(g\inv)>1$.  Then there are no arrows incident with $\xi$.
\end{itemize}
\end{rem}

\subsection{Axis trees}

Suppose for this subsection that $X$ is uniquely geodesic and that $g \in G$ translates along an axis.  We note that can interpret the scale of $g$ in terms of the amount of branching of $T$ going away from $\xi_+(g)$, where $T$ is a certain locally finite tree embedded in $X$.  The tree is also a close relative of the tree representation in the sense of scale theory, as introduced in \cite[\S4]{BaumgartnerWillis}.

\begin{defn}\label{defn:axis_tree}
Let $X$ be a uniquely geodesic NPC space and let $G$ be a \tdlc group acting properly and continuously by isometries on $X$ with open point stabilizers.  Suppose $\rho: \Rb \rightarrow X$ is a geodesic line and $g \in G$ is such that $g\rho(t) = \rho(t+|g|)$ for all $t \in \Rb$, and write $\xi = \lim_{t \rightarrow +\infty}\rho(t)$ (equivalently, $\xi$ is the attracting point of $g$).  Let $L_{G,\rho}$ be the set of axes of translation of $G$ that contain $\rho([t,+\infty))$ for some $t \in \Rb$, and let $T:=T_{G,\rho}$ be the union of $L_{G,\rho}$.  We take $T$ as a topological subspace of $X$, but equipped with the metric $d_T$, where for $x,y \in T$ we write $d_T(x,y)$ for the length of the shortest path from $x$ to $y$ within $T$, and equipped with $\xi$ as a distinguished end of $T$.

Given $x,y \in T$ we write $x \le_T y$ if there is a geodesic ray from $x$ to $\xi$ contained in $T$ that passes through $y$.  Take $x_0 \in T$ and define the branching function of $T$ based at $x_0$ as follows:
\[
\sigma_{T,x_0}: \Rb \rightarrow \Nb; \; \sigma_{T,x_0}(m) =
\begin{cases}
		|\{y \in T \mid y \le_T x_0, d_T(x_0,y) = m\}| &\mbox{if} \;  m > 0 \\ 
		1 &\text{otherwise}
\end{cases}.
\]
\end{defn}

\begin{figure}

\caption{The scale in terms of axes sharing a ray}
\label{fig:axis_tree}
\begin{center}
\begin{tikzpicture}[use Hobby shortcut]

\draw (4,-3) node{$gx_i = x_{i+1}; \quad s_G(g) = |G_{x_2,\xi_+(g)}:G_{x_1,\xi_+(g)}| = \sigma_{T,x_2}(|g|) = 9.$};

\tikzstyle{every node}=[circle,
                        inner sep=0pt, minimum width=3pt]

\draw[thick, name path = axis] (0,0) to (7,0);
\draw[name path = circle] (4.25,0) circle (2.0);
\draw (-1.5,0) node{$\xi_+(g)$};
\draw (-0.5,0) node{$\ldots$};
\draw (7.5,0) node{$\ldots$};
\draw (8.5,0) node{$\xi_+(g)$};

\foreach \x in {0,...,4} {
\ifthenelse{\x < 4}{
\draw (0.25+2*\x,0) node[draw=black, fill=black, minimum width=5pt]{};
\draw (2*\x,0.25) node{$x_{\x}$};
}{};
\foreach \y in {1,-1} {
\draw[thick, name path = a\x\y] (0+\x,2*\y) .. (0.5+\x,1*\y) .. (1+\x,0.6*\y) .. (1.5+\x,0.2*\y) ..(2+\x,0) -- (2.5+\x,0);
\ifthenelse{\x < 3}{\draw[name intersections = {of = circle and a\x\y}] (intersection-1) node[draw=black, fill=black]{}}{};
\draw[thick, name path = b\x\y] (0.5+\x,2*\y) .. (0.75+\x,1*\y) .. (1+\x,0.6*\y);
\ifthenelse{\x < 3}{\draw[name intersections = {of = circle and b\x\y}] (intersection-1) node[draw=black, fill=black]{}}{};
\draw[thick, name path = c\x\y] (0+\x,1.25*\y) .. (0.25+\x,1*\y) .. (1+\x,0.6*\y);
\ifthenelse{\x < 3}{\draw[name intersections = {of = circle and c\x\y}] (intersection-1) node[draw=black, fill=black]{}}{};
}
}
\end{tikzpicture}
\end{center}
\end{figure}
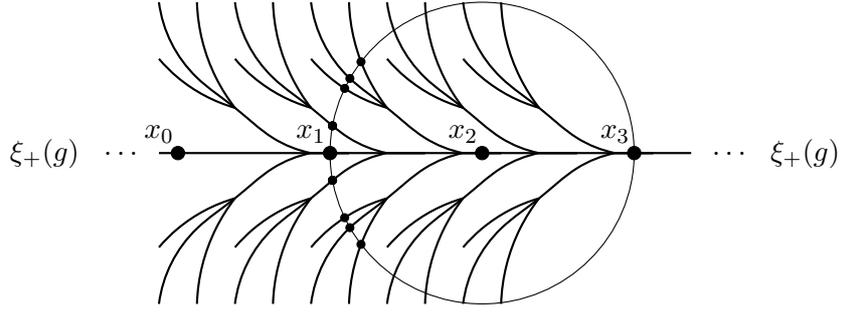

\begin{thm}[See Section~\ref{sec:axis_tree}]\label{thm:axis_tree}
Retain the hypotheses and notation of Definition~\ref{defn:axis_tree}.  Let $H$ be the stabilizer of $T_{G,\rho}$ in $G_\xi$, and let $S$ be the set of elements of $G$ with attracting end $\xi$ and possessing an axis that belongs to $L_{G,\rho}$.  Let $\lambda = \inf\{|h| \mid h \in S\}$.  Then $\lambda>0$ and the following holds.
\begin{enumerate}[(i)]
\item The group $H$ is an open subgroup of $G_{\xi}$ with $g \in S \subseteq H$, such that $s_H(h) = s_G(h)$ for all $h \in S$.
\item Every unbounded element of $H$ has an axis contained in $T$.  Moreover, there is a continuous homomorphism $\beta: H \rightarrow \Rb$ with image $\lambda \Zb$ such that $|\beta(h)| = |h|$ and $h \in S$ if and only if $\beta(h)<0$.
\item The space $T$ is embedded as a closed path-connected subspace of $X$.  As a metric space, $(T,d_T)$ is a geometric realization of a locally finite tree (in the combinatorial sense) with finitely many edge lengths; the sum of the edge lengths, taking one length for each $H$-orbit, is $\lambda$.
\item $H$ acts continuously, properly and cocompactly by isometries on $(T,d_T)$.
\item For all $h \in H$ and $x_0 \in T$, we have $s_H(h) = \sigma_{T,x_0}(-\beta(h))$.
\end{enumerate}
\end{thm}

An example of how the scale manifests geometrically in shown in Figure~\ref{fig:axis_tree}.  The large dots indicate an orbit $(x_i)$ of the hyperbolic element $g \in G$ along an axis $\gamma$.  The lines show a part of $T = T_{G,\gamma_+}$, and then starting from any point $x \in T$, we can calculate $s_G(g)$ by counting the number of points $y$ at distance $|g|$ from $x$ such that $y \le_T x$.

\subsection{Tidy subgroups}

As noted earlier, a key component of scale theory is the structure of compact open subgroups $U$ that minimize the index $|gUg\inv: U \cap gUg\inv|$, for a given $g \in G$.  In the case that $G$ is a \tdlc group acting properly and continuously by isometries on an NPC space $X$ and $g \in G$ is hyperbolic, we can give a variant of Willis's criteria for $U$ to be minimizing.

\begin{thm}[See Section~\ref{sec:tidy}]\label{thm:intro_tidy}
Let $X$ be an NPC space, let $G$ be a \tdlc group acting properly and continuously by isometries; let $g \in G$ be hyperbolic, with attracting point $\xi_+$ and repelling point $\xi_-$; and let $U$ be a compact open subgroup of $G$.  Then $U$ is minimizing for $g$ if and only if it satisfies the following conditions:
\begin{enumerate}
\item[$\GTA$] $U = U_{\xi_+}U_{\xi_-}$;
\item[$\GTp$] $g\inv U_{\xi_+}g \le U$;
\item[$\GTm$] $gU_{\xi_-}g\inv \le U$.
\end{enumerate}
Moreover, if $U$ is minimizing for $g$, then $U_{\xi_+} = U_{g+}$ and $U_{\xi_-} = U_{g-}$.
\end{thm}

In particular, although $G_{\xi_+}$, respectively $G_{\xi_-}$, is only an identity neighbourhood in $G$ when $s_G(g\inv)=1$, respectively $s_G(g)=1$, the product of $G_{\xi_+}$ and $G_{\xi_-}$ is an identity neighbourhood regardless of the scale.

\begin{cor}[See Section~\ref{sec:tidy}]\label{cor:intro_tidy_neighbourhood}
Let $X$ be an NPC space, let $G$ be a \tdlc group acting properly and continuously by isometries and let $g \in G$ be hyperbolic, with attracting and repelling ends $\xi_+$ and $\xi_-$ respectively, and let $U \le G$ be open.  Then the product $U_{\xi_+}U_{\xi_-}$ is a neighbourhood of the identity in $G$.
\end{cor}

Conditions $\GTp$ and $\GTm$ will occur automatically if, for example, $g$ has an axis $L$ passing through a point with open stabilizer, and $U$ is the fixator in $G$ of some bounded segment of $L$.  In this context, we obtain a source of minimizing subgroups for $g$.

\begin{cor}[See Section~\ref{sec:tidy}]\label{cor:intro_tidy_axis}
Let $X$ be an NPC space, let $G$ be a \tdlc group acting properly and continuously by isometries and let $g \in G$ be hyperbolic.  Suppose that the image of the isometric embedding $\gamma: \Rb \rightarrow X$ is an axis for $g$, and that $G_{\gamma(0)}$ is open.  Then there exists $t_0 \ge 0$ such that for all $t \ge t_0$, the group $U = G_{\gamma(0),\gamma(t)}$ is minimizing for $g$.
\end{cor}

Example~\ref{ex:small_tidy} below shows why there is no bound on $t_0$ in the statement of Corollary~\ref{cor:intro_tidy_axis}; in particular, $t_0$ can be an arbitrarily large multiple of the translation length of $g$.  If $G$ is acting on a complete CAT(0) space with open point stabilizers, then Corollary~\ref{cor:intro_tidy_axis} applies to all hyperbolic elements, since an axis always exists in this context.  (The specific case where $G$ is the isometry group of a regular locally finite tree was considered already in \cite[\S3]{Willis94}, with similar conclusions.)  In the $\delta$-hyperbolic setting, $g$ does not necessarily have an axis; however, if for example $g$ is a hyperbolic isometry of a locally finite $\delta$-hyperbolic graph, then some positive power of $g$ admits an axis (\cite[Theorem~3.1]{Bywaters}).

\subsection{A geometric interpretation of the contraction group}

At the end of the article we turn to another major component of scale theory, the contraction group.  This can be defined for $g$ acting on $G$ itself, or more generally with respect to the quotient space $G/H$ whenever $H$ is a closed $\grp{g}$-invariant subgroup\footnote{Strictly speaking, the set $\con(g/H)$ is not always a group at this level of generality, however in practice we will only consider it in the case that it is a group, so the term `contraction group' is justifiable.}:
\begin{align*}
\con_G(g) &:= \{h \in G \mid g^nhg^{-n} \rightarrow 1 \text{ as } n \rightarrow +\infty\}; \\
\con_G(g/H) &:= \{h \in G \mid g^nhg^{-n}H \rightarrow H \text{ as } n \rightarrow +\infty\}.
\end{align*}

In general $\con_{G}(g)$ need not be closed in $G$.  Indeed, there are significant restrictions on the structure of closed contraction groups in general due to Gl\"{o}ckner--Willis (\cite{GlocknerWillis}, \cite{GlocknerWillis2}), and Caprace--De Medts (\cite{CDM}) found strong global consequences of closed contraction groups in the case that $G$ is acting boundary-transitively on a locally finite tree.  Hence any more generally-applicable approach must take account of the difference between a contraction group and its closure.

Now suppose $X$ is a metric space on which $G$ acts properly and continuously by isometries, let $g \in G$ and let $Y$ be a nonempty subspace, with $K:= \Fix_G(Y)$.  If $g$ has a bounded orbit, then it is contained in a compact subgroup, so $\con_G(g/K)$ is trivial. On the other hand if $g$ has unbounded orbits, then $h \in \con_G(g/K)$ if and only if, for all $y \in Y$, we have $d(hg^{-n}y,g^{-n}y) \rightarrow 0$ as $n \rightarrow +\infty$ (see Lemma~\ref{lem:basic_contraction}).

There is a more interesting characterization of elements of $\con_G(g/K)$ under the following assumptions, which hold in practice for many cases of interest for the theory of \tdlc groups.

\begin{hyp}\label{hyp:locally_discrete}
Let $X$ be a proper NPC space, and let $G$ be a \tdlc group acting properly on $X$ by isometries.  We suppose also that the action of $G$ on $X$ is \defbold{locally discrete}, meaning that around each point there is a ball of positive radius whose fixator in $G$ is open.
\end{hyp}

For example, Hypothesis~\ref{hyp:locally_discrete} holds in the case that $X$ is a locally finite CAT(0) $M_{\kappa}$-polyhedral complex (in the sense of \cite[Chapter I.7]{BH}, with $\kappa \le 0$) and $G$ is a closed subgroup of the cellular isometry group, where the latter carries the permutation topology acting on the set of cells.  Under Hypothesis~\ref{hyp:locally_discrete}, we can describe the contraction group as follows.

\begin{defn}
Let $X$ be a proper NPC space, let $Y$ be a closed convex subspace, let $\xi \in \partial X$ and let $Z$ be any subset of $X$.  We say that $Z$ is an \defbold{absorbing set for $\xi$ within $Y$} if for some (equivalently, all) rays $\rho: \Rb_{\ge 0} \rightarrow X$ representing $\xi$ and all $t \gg 0$, we have $\{y \in Y \mid d(\rho(t),y) \le r_t\} \subseteq Z$ such that $r_t \rightarrow +\infty$ as $t \rightarrow +\infty$.
\end{defn}

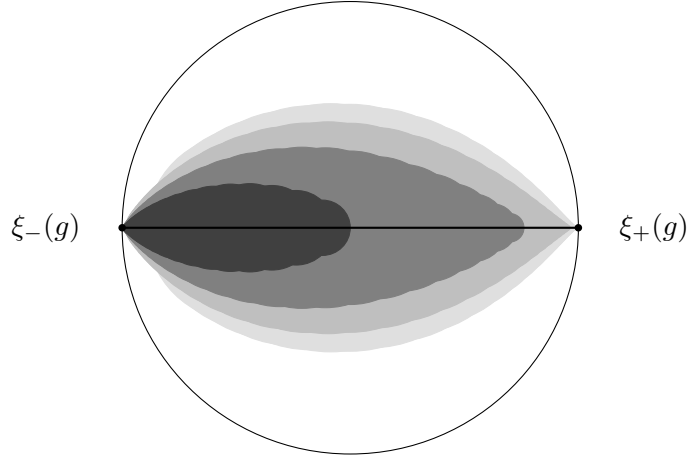
\begin{figure}

\caption{Some translates of an absorbing set in $\mathbb{H}^2$}
\label{fig:absorbing}
\begin{center}
\begin{tikzpicture}[use Hobby shortcut]

\tikzstyle{every node}=[circle,
                        inner sep=0pt, minimum width=2pt]

\foreach \x in {0,...,30} {
\tikzmath{real \r; \r = sqrt(\x);};
\tikzmath{real \y; \y = (3*tanh((\x-24+\r)/8)+3*tanh((\x-24-\r)/8))/2;};
\tikzmath{real \r2; \r2 = (3*tanh((\x-24+\r)/8)-3*tanh((\x-24-\r)/8))/2;};
\draw[lightgray!50, fill=lightgray!50] (-\y,0) circle (\r2);
}

\foreach \x in {0,...,30} {
\tikzmath{real \r; \r = sqrt(\x);};
\tikzmath{real \y; \y = (3*tanh((\x-16+\r)/8)+3*tanh((\x-16-\r)/8))/2;};
\tikzmath{real \r2; \r2 = (3*tanh((\x-16+\r)/8)-3*tanh((\x-16-\r)/8))/2;};
\draw[lightgray, fill=lightgray] (-\y,0) circle (\r2);
}

\foreach \x in {0,...,30} {
%\ifthenelse{\x = 0}{ 
%\draw (-{3*tanh((\x-8)/8)},-0.5) node{$gx$};
%\draw (-{3*tanh((\x-8)/8)},0) node[draw=black, fill=black, minimum width=5pt]{};
%}{};
\tikzmath{real \r; \r = sqrt(\x);};
\tikzmath{real \y; \y = (3*tanh((\x-8+\r)/8)+3*tanh((\x-8-\r)/8))/2;};
\tikzmath{real \r2; \r2 = (3*tanh((\x-8+\r)/8)-3*tanh((\x-8-\r)/8))/2;};
\draw[gray, fill=gray] (-\y,0) circle (\r2);
}

\foreach \x in {0,...,30} {
%\ifthenelse{\x = 0}{ 
%\draw (-{3*tanh((\x)/8)},-0.5) node{$x$};
%\draw (-{3*tanh((\x)/8)},0) node[draw=black, fill=black, minimum width=5pt]{};
%}{};
\tikzmath{real \r; \r = sqrt(\x);};
\tikzmath{real \y; \y = (3*tanh((\x+\r)/8)+3*tanh((\x-\r)/8))/2;};
\tikzmath{real \r2; \r2 = (3*tanh((\x+\r)/8)-3*tanh((\x-\r)/8))/2;};
\draw[darkgray, fill=darkgray] (-\y,0) circle (\r2);
}

\draw[thick, name path = axis] (-3,0) node[draw=black, fill=black]{} to (3,0) node[draw=black, fill=black]{};
\draw[name path = circle] (0,0) circle (3.0);
\draw (-4,0) node{$\xi_-(g)$};
\draw (4,0) node{$\xi_+(g)$};

\end{tikzpicture}
\end{center}
\end{figure}

\begin{thm}[See Section~\ref{sec:contraction}]\label{thm:intro_contraction}
Assume Hypothesis~\ref{hyp:locally_discrete}.  Let $g \in G$, let $C$ be a compact subset of $G$, let $Y$ be a nonempty closed convex $\grp{g}$-invariant subspace of $X$ and let $K = \Fix_G(Y)$.  If $g$ is bounded then $\con_G(g/K) = K$.  If $g$ is hyperbolic, then the following are equivalent:
\begin{enumerate}[(a)]
\item $C \subseteq \con_G(g/K)$;
\item The set of fixed points of $\grp{C}$ is an absorbing set for $\xi_-(g)$ within $Y$.
\end{enumerate}
\end{thm}

In particular, in this context the contraction group of a hyperbolic element $g$ only depends on its repelling point.

The prototypical example of an absorbing set is a horoball centred at $\xi$, but an absorbing set need not contain any horoball, since $r_t$ is allowed to tend to $+\infty$ arbitrarily slowly.  Figure~\ref{fig:absorbing} shows (in different shades) some translates of an absorbing set for a hyperbolic translation $g$, where $r_t = O(\sqrt{t})$.  This complication is unavoidable in general: if one considers the case that $X$ is a regular locally finite tree, $G = \mathrm{Isom}(X)$ and $g \in G$ translating along a geodesic line $\gamma$, then for any function $t \mapsto r_t$ that tends to $+\infty$, it is easy to construct $h \in \con_G(g)$ such that the largest ball around $\gamma(-t)$ that is fixed pointwise by $h$ only grows at a rate of $O(r_t)$ as $t \rightarrow +\infty$.  So in the situation of Theorem~\ref{thm:intro_contraction}, it can happen that there is no uniform choice of absorbing set $Z$ such that every element of $\con_G(g)$ fixes pointwise some $\mathrm{Isom}(X)$-translate of $Z$.  This is a manifestation of the failure of $\con_G(g)$ to be closed in general.

Some of the complications described in the previous paragraph disappear if we pass to a sufficiently small $\grp{g}$-invariant subspace.  Although contraction groups are not closed in general, it was shown by Baumgartner--Willis (extended to the non-metrizable case by Jaworski) that there is always a compact $\grp{g}$-invariant subgroup $K$ of $G$ such that $\con_G(g/K)$ is closed.  The smallest such $K$ is the \defbold{nub} $\nub_G(g)$ of $g$, which has several equivalent characterizations: for example, it is the intersection of all minimizing subgroups for $g$.  We can extract from Theorem~\ref{thm:intro_contraction} some equivalent conditions for when $Y$ is a \defbold{closed contraction space} for $g$, meaning that $\con_{G}(g/\Fix_G(Y))$ is closed.

\begin{cor}[See Section~\ref{sec:contraction}]\label{cor:intro_contraction_closed}
Assume Hypothesis~\ref{hyp:locally_discrete}.  Let $g \in G$ be hyperbolic, let $Y$ be a nonempty closed convex $\grp{g}$-invariant subspace of $X$ and let $K = \Fix_G(Y)$.  Then the following are equivalent:
\begin{enumerate}[(a)]
\item $\con_G(g/K)$ is closed;
\item $\nub_G(g)$ fixes $Y$ pointwise;
\item For some (equivalently all) $x \in X$, the set of fixed points of $\con_G(g/K)_x$ is absorbing for $\xi_-(g)$ within $Y$.
\end{enumerate}
\end{cor}

If $Y$ is a closed contraction space and $V$ is a minimizing subgroup for $G$, we can say a bit more about $X^W$ where $W = V \cap \con_G(g/\Fix_G(Y))$.

\begin{cor}[See Section~\ref{sec:contraction}]\label{cor:intro_contraction_tidy}
Assume Hypothesis~\ref{hyp:locally_discrete}.  Let $g \in G$ be hyperbolic, let $Y$ be a nonempty closed convex $\grp{g}$-invariant subspace of $X$ and let $K = \Fix_G(Y)$.  Let $V$ be minimizing for $g$, let $W = V \cap \con_G(g/K)$ and let $Z = X^W \cap Y$.  Suppose $\con_G(g/K)$ is closed.  If $s_G(g^{-1})=1$ then $\con_G(g/K)=K$, so $Z = Y$.  Otherwise, $Z$ is a closed convex subspace such that
\[
Z \subseteq gZ; \quad \bigcup_{n \ge 0} g^nZ = Y; \quad \bigcap_{n \le 0} g^nZ = \emptyset.
\]
\end{cor}

\begin{rem}
Assume Hypothesis~\ref{hyp:locally_discrete}, that $G$ acts faithfully and that every compact subgroup of $G$ has a fixed point in $X$; let $g \in G$ be axial.  One can obtain a space $Z$ as in Corollary~\ref{cor:intro_contraction_tidy} as follows.  Let $Y = X^{\nub_G(g)}$ and $K = \Fix_G(Y)$.  Take a point $x$ on an axis of $g$ and set $W_0 = \con_G(g)_x$.  We deduce from Corollary~\ref{cor:intro_tidy_axis} that $W_0 = V \cap \con_G(g)$ where $V$ is a minimizing subgroup for $g$; one then has $\ol{W_0} \ge \nub_G(g)$, so
\[
X^{W_0} = X^{\ol{W_0}} \subseteq Y.
\]
Given Corollary~\ref{cor:intro_contraction_closed}, one sees that $W = V \cap \con_G(g/K)$ is the closure in $G$ of $W_0K$, so in fact $X^{W_0} = X^W$ and hence $Z = X^{W_0}$ is as in Corollary~\ref{cor:intro_contraction_tidy}. The set $X^{W_0}$ can be thought of as the intersection of all absorbing sets $Z'$ for $\xi_-(g)$ within $X$, such that $x \in Z'$ and $Z'$ is realized as the fixed points of some element of $G$.  This intersection is itself an absorbing set for $\xi_-(g)$ within $X$ if and only if $\con_G(g)$ is closed; one recovers $X^{\nub_G(g)}$ as the union of all $\grp{g}$-translates (or indeed the union of translates by positive powers of $g$) of $X^{W_0}$.

For example, if $X$ is a regular locally finite tree and $G = \Isom(X)$, then every absorbing set that is the convex hull of a set of vertices occurs as $X^h$ for some $h \in \con_G(g)$, and consequently the space $Z$ we obtain is just a ray representing $\xi_-(g)$, with $\bigcup_{n \ge 0}g^nZ$ being the axis of $g$.  At the other extreme, if $G = \mathrm{PGL}_2(\Qb_p)$ acting on its Bruhat--Tits tree $X$, then $\con_G(g)$ is closed, $Z$ is a horoball, and $\bigcup_{n \ge 0}g^nZ$ is the whole tree: see Example~\ref{ex:SL2}.
\end{rem}

\paragraph{Acknowledgements.} I thank Michal Ferov and George Willis for productive discussions on the topic of this article.

\section{Preliminaries}

\subsection{Scale theory}

Let $G$ be a \tdlc group and let $\mc{COS}(G)$ be the set of compact open subgroups of $G$.

Let $g \in G$.  The \defbold{scale} of $g$ is
\[
s_G(g) := \min \{ |gUg\inv: gUg\inv \cap U| \mid U \in \mc{COS}(G)\}.
\]
A subgroup $U \in \mc{COS}(G)$ that achieves the minimum is called \defbold{mimimizing} for $g$.

\begin{thm}[{\cite[Theorem~3.1]{WillisFurther}}]\label{thm:Willis}
Let $G$ be a \tdlc group and let $g \in G$.  Then $U$ is minimizing for $g$ if and only if $U$ is \defbold{tidy} for $g$, which means it has the following two properties:
\item[$\TA$] $U = U_{g+}U_{g-}$, where $U_{g+} = \bigcap_{n \ge 0}g^nUg^{-n}$ and $U_{g-} = \bigcap_{n \le 0}g^nUg^{-n}$;
\item[$\TB$] $U_{g++}:=\bigcup_{n \ge 0}g^nU_{g+}g^{-n}$ is closed in $G$.
\end{thm}

If we replace $g$ with $g^{-1}$, condition $\TA$ is unaffected; by \cite[Lemma 3(b)]{Willis94}, condition $\TB$ is also preserved.  We will use the following corollary without further comment.

\begin{cor}
Let $G$ be a \tdlc group and let $g \in G$.  Then $U$ is minimizing for $g$ if and only if $U$ is minimizing for $g^{-1}$.
\end{cor}

The modular function can be recovered from the scale, as follows: given a right-invariant Haar measure $\mu$ and $U \in \mc{COS}(G)$, then
\[
\Delta_G(g) = \frac{\mu(gU)}{\mu(U)} = \frac{|gUg\inv: gUg\inv \cap U|}{|U: gUg\inv \cap U|} = \frac{s_G(g)}{s_G(g\inv)}.
\]
In particular, $G$ is unimodular if and only if, for all $g \in G$, we have $s_G(g) = s_G(g\inv)$.

Within a tidy subgroup $U$, the subgroups $U_{g+}$ and $U_{g-}$ are characterized by which forward and backward orbits of $g$ are bounded.

\begin{lem}\label{lem:parabolic_tidy}
Let $G$ be a \tdlc group, let $g \in G$ and let $U$ be tidy for $g$.  Then
\[
U \cap \para_G(g\inv) = U_{g+} \text{ and } U \cap \para_G(g) = U_{g-}.
\]
\end{lem}

\begin{proof}
It is clear that $U_{g+} \le U \cap \para_G(g)$, so
\[
U \cap \para_G(g\inv)  = U_{g+}(U_{g-} \cap \para_G(g\inv)).
\]
By \cite[Lemma~9]{Willis94}, every element of $U_{g-} \cap \para_G(g\inv)$ also belongs to $U_{g+}$, so $U \cap \para_G(g\inv)  = U_{g+}$.  The proof that $U \cap \para_G(g) = U_{g-} = U_{g-}$ is similar.
\end{proof}

The \defbold{contraction group} of $g$ is
\[
\con_G(g) := \{x \in G \mid g^nxg^{-n} \rightarrow 1 \text{ as } n \rightarrow +\infty\}.
\]
More generally, given $K \le G$ closed, we define
\[
\con_G(g/K) := \{x \in G \mid g^nxg^{-n}K \rightarrow K \text{ as } n \rightarrow +\infty\},
\]
where the convergence is in the coset space $G/K$.

In the dynamics of conjugation by a cyclic subgroup $\grp{g}$, the main distinctions are between scale $1$ and scale greater than $1$, considering the scale of $g$ and of $g\inv$.

\begin{lem}[{\cite[Proposition~3.24]{BaumgartnerWillis}}]\label{lem:BW_scale_one}
The following are equivalent:
\begin{enumerate}[(i)]
\item $s_G(g)=1$;
\item There exists $U \in \mc{COS}(G)$ such that $gUg\inv \le U$;
\item $\con_G(g\inv)$ has compact closure;
\item $\para_G(g)$ is open.
\end{enumerate}
\end{lem}

In particular, we always have $s_{\para_G(g)}(g)=1$.  The scale of $g$ can in fact be obtained as the modular function of $g$ as an element of $\para_G(g\inv)$.

\begin{lem}\label{lem:anisotropic}
Let $G$ be a \tdlc group and let $g \in G$.  Then the following are equivalent:
\begin{enumerate}[(i)]
\item We have $\con(g) = \con(g\inv) = \triv$;
\item For every open subgroup $H$ of $G$, then $\bigcap_{n \in \Zb}g^nHg^{-n}$ is open.
\end{enumerate}
\end{lem}

We say $g \in G$ is \defbold{uniscalar} if $s(g) = s(g\inv) = 1$, in other words, $g$ normalizes some compact open subgroup of $G$.

\begin{thm}[{\cite[Theorem 3.8]{BaumgartnerWillis}, \cite[Theorem 1]{Jaw}}]\label{thm:bw:relative_contraction}Let $G$ be a \tdlc group, let $g \in G$ and let $H$ be a closed subgroup of $G$ such that $gHg\inv = H$.  Then $\con_{G}(g/H) = \con_G(g)H$.\end{thm}

The \defbold{nub} $\nub_G(g)$ is the intersection of all tidy subgroups for $g$ in $G$.

\begin{lem}[{\cite[Corollary~3.30]{BaumgartnerWillis}}]\label{lem:contraction_closure}
Let $G$ be a \tdlc group and let $g \in G$.  Then
\[
\ol{\con_G(g)} = \con_G(g)\nub_G(g).
\]
\end{lem}

The following is an easy consequence of Theorem~\ref{thm:bw:relative_contraction} and Lemma~\ref{lem:contraction_closure}.

\begin{cor}\label{cor:contraction_closure}
Let $G$ be a \tdlc group, let $g \in G$ and let $K \le G$ be compact and $\grp{g}$-stable.  Then $\con_G(g/K)$ is closed if and only if $K \ge \nub_G(g)$; in particular,
\[
\ol{\con_G(g)} = \con_G(g/\nub_G(g)).
\]
\end{cor}

\begin{lem}\label{lem:scale_subgroup}
Let $G$ be a \tdlc group, let $\alpha \in \Aut(G)$ and let $H$ be a closed $\alpha$-stable subgroup of $G$.  Then $s_H(\alpha) \le s_G(\alpha)$; if $\con(\alpha^{-1}) \le H$ then equality holds.
\end{lem}

\begin{proof}
In \cite[Proposition 4.3]{WillisFurther} it is shown that $s_H(\alpha) \le s_G(\alpha)$.  By \cite[Proposition~3.21]{BaumgartnerWillis} we have $s_G(\alpha) = s_K(\alpha)$ for $K = \ol{\con(\alpha^{-1})}$, so if $H \ge K$ then $s_G(\alpha) = s_H(\alpha)$.
\end{proof}

\begin{cor}\label{cor:scale_subgroup:modular}
Let $G$ be a \tdlc group, let $\alpha \in \Aut(G)$ and let $H = \para_G(\alpha\inv)$.  Then
\[
s_G(\alpha) = s_H(\alpha) =  \Delta_H(\alpha).
\]
\end{cor}

\begin{proof}
The first equality is immediate from Lemma~\ref{lem:scale_subgroup}, while the second follows from the fact that $s_H(\alpha\inv)=1$ (by Lemma~\ref{lem:BW_scale_one}) and the fact that $\Delta_H(\alpha) = s_H(\alpha)/s_H(\alpha\inv)$.
\end{proof}

\begin{lem}[{\cite[Lemma 4.3]{CRW-TitsCore}}]\label{lem:tidy_stability}
Let $G$ be a \tdlc group, let $g \in G$ and let $U$ be tidy above for $g$.  Then for every $u \in U \cap g Ug\inv$ there is $r \in U$ such that, for every $n \in \Zb$,
\[
(ug)^n \in rg^n U.
\]
\end{lem}

\begin{lem}\label{lem:uniform_contract}
Let $G$ be a locally compact group, let $C \subseteq G$ and $K \le G$ be compact, and suppose that $g \in \N_G(K)$ is such that $C \subseteq \con_G(g/K)$.  Then given an identity neighbourhood $O$ such that $O = OK$, there is some $n_0$ such that $g^nCg^{-n} \subseteq O$ for all $n \ge n_0$.
\end{lem}

\begin{proof}
Let $O_2$ be a compact identity neighbourhood such that $O_2O_2 \subseteq O$, and let $O_- = \bigcap_{n \ge 0}g^{-n}O_2g^n$.  Since $C \subseteq \con_G(g/K)$, we see that $C$ is contained in the increasing union $\bigcup_{n \ge 0}g^{-n}O_-g^{n}$.  By the Baire Category Theorem, there is some $n_1$ such that the interior of $g^{-n_1}O_-g^{n_1} \cap C$ is nonempty, so it contains $hU$ for some $h \in G$ and open identity neighbourhood $U$.  Since $C$ is compact we have $C \subseteq \bigcup^k_{i=1}h_iU$ for some $h_1,\dots,h_k \in C$.  We see that also $h_1h\inv,\dots,h_kh\inv \in C$, so there is some $n_0 \ge n_1$ such that $h_ih\inv \in g^{-n_0}O_-g^{n_0}$ for $1 \le i \le k$.  Then for all $n \ge n_0$ and $c \in C$, we have
\[
c \in \bigcup^k_{i=1}(h_i h\inv)hU \subseteq (g^{-n}O_-g^n)(g^{-n}O_-g^n) \subseteq O_2O_2 \subseteq O. \qedhere
\]
\end{proof}

\subsection{Nonpositively curved spaces}

A \defbold{geodesic} of a metric space $X$ is an isometric embedding of a convex subset of the Euclidean line $\Rb$ into $X$; we say $X$ is a \defbold{geodesic space} if for all $x,y \in X$ there is a geodesic $\gamma: [0,d(x,y)] \rightarrow X$ such that $\gamma(0)=x$ and $\gamma(d(x,y))=y$.

Let $X$ be a geodesic space.  Two geodesic rays $\gamma_1,\gamma_2: \Rb_{\ge 0} \rightarrow X$ are \defbold{asymptotic} to one another, written $\gamma_1 \sim \gamma_2$, if $d(\gamma_1(t),\gamma_2(t))$ is bounded over all $t$.  An equivalence class of rays is called a \defbold{point at infinity} or \defbold{boundary point} of $X$, and the set of all points at infinity is denoted $\partial X$, the \defbold{boundary} of $X$.  Write $\gamma(\infty)$ for the equivalence class of the geodesic ray $\gamma$.

Given a geodesic line $\gamma: \Rb \rightarrow X$, there are two associated rays: $\gamma_+,\gamma_-:\Rb_{\ge 0} \rightarrow X$, where $\gamma_+(t) = \gamma(t)$ and $\gamma_-(t) = \gamma(-t)$.  We then define $\gamma(+\infty) = \gamma_+(\infty)$ and $\gamma(-\infty) = \gamma_-(\infty)$.  Note that since $\gamma$ is geodesic, the rays $\gamma_+$ and $\gamma_-$ move away from each other, so $\gamma(+\infty) \neq \gamma(-\infty)$.  Two geodesic lines $\gamma,\gamma'$ are \defbold{asymptotic} if $\gamma_+ \sim \gamma'_+$ and $\gamma_- \sim \gamma'_-$, in other words $\gamma(+\infty) = \gamma'(+\infty)$ and $\gamma(-\infty) = \gamma'(-\infty)$.

A \defbold{geodesic triangle} is a triple of geodesics $\gamma_i: [0,a_i] \rightarrow X$ ($i \in \{0,1,2\}$) such that $\gamma_{i+1}(0) = \gamma_i(a_i)$ (reading subscripts modulo $3$).  A \defbold{generalized geodesic triangle} is defined analogously, except the geodesics forming the sides can have infinite length, with some of the corners being points at infinity.  Given a geodesic triangle $(\gamma_0,\gamma_1,\gamma_2)$, a \defbold{comparison triangle} is a geodesic triangle $(\gamma'_0,\gamma'_1,\gamma'_2)$ in the Euclidean plane $\Rb^2$ with the same lengths.  A (generalized) geodesic triangle is \defbold{$\delta$-thin} if for every point $x$ on any of the sides (not counting the corners), there is a point $y$ lying on a different side of the triangle, such that $d(x,y) \le \delta$.

Let $X$ be a geodesic space.
\begin{itemize}
\item $X$ is \defbold{CAT(0)} if for any geodesic triangle and comparison triangle as above, one has $d_X(\gamma_1(s),\gamma_2(t)) \le d_{\Rb^2}(\gamma'_1(s),\gamma'_2(t))$.
\item $X$ is \defbold{$\delta$-hyperbolic} (for some parameter $\delta \in \Rb_{\ge 0}$) if every geodesic triangle is $\delta$-thin.
\item We say $X$ is an \defbold{NPC space} if either $X$ is complete and CAT(0), or else $X$ is proper and $\delta$-hyperbolic.
\end{itemize} 

We say a pair of boundary points $\xi_1,\xi_2$ are \defbold{opposite} one another (or \defbold{collinear}) if there is a geodesic line $\gamma$ with $\gamma(-\infty) = \xi_1$ and $\gamma(+\infty) = \xi_2$.  In a $\delta$-hyperbolic space $X$, any two distinct points of $\partial X$ are opposite.  This is not true in general for complete CAT(0) spaces: taking for example the space $\Rb^n$ for $n\ge 2$, one sees that $\partial \Rb^n$ is infinite but every point in $\partial \Rb^n$ has a unique opposite.   Nevertheless, in both cases all boundary points can be `seen' from any base point.

\begin{lem}\label{lem:visibility}\
\begin{enumerate}[(i)]
\item(\cite[Proposition II.8.2]{BH}) Let $X$ be a complete CAT(0) space, let $x \in X$ and let $\xi \in \partial X$.  Then there is a unique geodesic ray $\gamma$ such that $\gamma(0)=x$ and $\gamma(\infty) = \xi$.
\item(\cite[Chapitre 2 Proposition 2.1]{CDP}) Let $X$ be proper $\delta$-hyperbolic space, let $x \in X$ and let $\xi \in \partial X$.  Then there is a geodesic ray $\gamma$ such that $\gamma(0)=x$ and $\gamma(\infty) = \xi$.  Moreover, any two distinct points in $\partial X$ are opposite.
\end{enumerate}
\end{lem}

If $X$ is a proper $\delta$-hyperbolic space, convergence at infinity is defined as follows.  We define the Gromov product with respect to some base point $e$,
\[
(y|z)_{e} = \frac{1}{2}(d(e,y)+d(e,z)-d(y,z));
\]
observe if we replace $e$ with some other base point $e'$, then $|(y|z)_e-(y|z)_{e'}| \le d(e,e')$.  Say that the sequence $(x_n)$ is \defbold{Cauchy--Gromov} if $(x_n|x_m)_e$ tends to $+\infty$ as $n,m \rightarrow +\infty$.  In particular, one can check that if $n \mapsto x_n$ is a quasi-isometric embedding of $\Nb$ into $X$, then $(x_n)$ is Cauchy--Gromov.  A pair of Cauchy--Gromov sequences $(x_n)$ and $(y_n)$ are defined to be equivalent if $(x_n|y_m)_e$ tends to $+\infty$ as $n,m \rightarrow +\infty$.  Write $\partial_s X$ for the set of equivalence classes; the next lemma allows us to identify $\partial_s X$ with $\partial X$ in a natural way.

\begin{lem}[{See \cite[Lemma~H.3.13]{BH}}]
If $X$ is a proper $\delta$-hyperbolic space then there is a unique bijection from $\partial_s X$ to $\partial X$, such that for each geodesic ray $\gamma: \Rb_{\ge 0} \rightarrow X$, the equivalence class of $(\gamma(n))$ is sent to $\gamma(\infty)$.
\end{lem}

In a proper $\delta$-hyperbolic space, quasi-geodesics stay close to geodesics.

\begin{lem}\label{lem:close_geodesic}
Let $X$ be a proper $\delta$-hyperbolic space, let $(x_n)_{n \in \Zb}$ be a bi-infinite sequence of points such that $n \mapsto x_n$ is quasi-geodesic and let $\gamma$ be a geodesic ray such that $(x_n)_{n \ge 0} \sim (\gamma(n))_{n \ge 0}$.  Then $(x_n)_{n \ge 0}$ stays within a bounded distance of the image of $\gamma$.
\end{lem}

\begin{proof}
By Lemma~\ref{lem:visibility}(ii) there is a geodesic line $\eta$ such that $(x_n) \rightarrow \eta(\pm\infty)$ as $n \rightarrow \pm \infty$.  By \cite[Chapitre~3 Th\'{e}or\`{e}me 3.1]{CDP}, the sequence $(x_n)$ stays within some distance $\kappa$ of $\eta$; given the orientation of $\eta$ relative to $(x_n)$, in fact for $n>0$ sufficiently large, we have $d(x_n,\eta(s_n))\le \kappa$ where $s_n \ge 0$.  Meanwhile, by \cite[Lemma~H.3.3]{BH}, the geodesic ray $\eta_+$ eventually stays within a distance $5\delta$ of $\gamma$.  Thus for $n \gg 0$ there is $t_n \in \Rb_{\ge 0}$ such that $d(x_n,\gamma(t_n)) \le \kappa + 5\delta$.
\end{proof}

Both complete CAT(0) spaces and hyperbolic spaces have restrictions on the thickness of generalized triangles.

\begin{lem}\label{lem:decreasing_distance}
Let $X$ be a complete CAT(0) space and let $\gamma_1$ and $\gamma_2$ be geodesic rays.  Then $\gamma_1 \sim \gamma_2$ if and only if
\[
\forall t \ge 0: d(\gamma_1(t),\gamma_2(t)) \le d(\gamma_1(0),\gamma_2(0)).
\]
\end{lem}

\begin{proof}
If $\gamma_1$ and $\gamma_2$ are not asymptotic, then clearly there is some $t > 0$ such that $d(\gamma_1(t),\gamma_2(t)) > d(\gamma_1(0),\gamma_2(0))$.  On the other hand if $\gamma_1$ and $\gamma_2$ are asymptotic, then $f(t) = d(\gamma_1(t),\gamma_2(t))$ is a bounded continuous function from $[0,\infty)$ to $[0,\infty)$ that is also convex; see \cite[Proposition II.2.2]{BH}.  In particular, $f$ must be weakly decreasing, as desired.
\end{proof}

\begin{lem}\label{lem:thin_omega_triangle}
Let $X$ be a proper $\delta$-hyperbolic space.  Then generalized geodesic triangles in $X$ are $24\delta$-thin.  Moreover, given two geodesics $\gamma_1$ and $\gamma_2$ with common limit $\xi$, then
\[
d(\gamma_1(t),\gamma_2(t)) \le 48\delta+r
\]
for all $t \ge 24\delta+r$, where $r = d(\gamma_1(0),\gamma_2(0))$.
\end{lem}

\begin{proof}
For the assertion about generalized geodesic triangles, see \cite[Chapitre 2 Proposition 2.2]{CDP}.  Now let $\gamma_1$ and $\gamma_2$ be geodesics converging to the boundary point $\xi$, and let $x_1 = \gamma_1(0)$, $x_2 = \gamma_2(0)$ and $r = d(x_1,x_2)$.  Then the triangle formed by $x_1$, $x_2$ and $\xi$ is $24\delta$-thin, so for $t_1 > 24\delta+r$, there exists $u \ge 0$ such that $d(\gamma_1(t),\gamma_2(u)) \le 24\delta$.  By the triangle inequality, $u$ cannot differ from $t$ by more than $24\delta+r$, so $d(\gamma_1(t),\gamma_2(t)) \le 48\delta + r$.
\end{proof}

\subsection{Types of isometries}\label{sec:types}

Given a metric space $X$ and a subset $Y$, we will say an isometry $g$ \defbold{stabilizes} $Y$ if $\{gy \mid y \in Y\} = Y$, and $g$ \defbold{fixes} $Y$ if $\forall y \in Y: gy=y$.  We define the \defbold{fixator} $\Fix_G(Y)$ and \defbold{stabilizer} $G_Y$ of $Y$ analogously; however, 
\[
G_{x_1,\dots,x_n} := \bigcap^n_{i=1}G_{x_i}= \Fix_G(\{x_1,\dots,x_n\}).
\]
Given an isometry or set of isometries $H$, we write $X^H$ for the set of fixed points of $H$, in other words the largest subspace of $X$ fixed by $H$.

Let $X$ be a metric space and let $g \in \Isom(X)$.  We define the \defbold{translation length} $|g|$ of $g$ to be 
\[
|g| := \liminf_{n \rightarrow +\infty} \frac{d(g^nx,x)}{n} \quad (x \in X).
\]
It is easy to see that $|g|$ does not depend on the choice of $x$; using the triangle inequality, one sees that $|g| \le d(gx,x)$ for all $x \in X$.  We then define $\Min(g) := \{ x \in X \mid d(gx,x) = |g|\}$; notice that $|g| = |g\inv|$ and $\Min(g) = \Min(g\inv)$.

A group $G$ of isometries of a metric space $X$ is called \defbold{bounded} if for all (equivalently, some) $x \in X$, the orbit $Gx = \{gx \mid g \in G\}$ remains within a bounded distance of $x$.  We say $g \in \Isom(X)$ is bounded if $\grp{g}$ is bounded.  In particular, every bounded isometry $g$ has zero translation length.

A \defbold{QI-hyperbolic} isometry is an isometry $g$ of a metric space $X$ such that we have a quasi-isometric embedding of $\Zb$ into $X$ via $n \mapsto g^nx$ for some (equivalently, any) $x \in X$.  In particular, note that every QI-hyperbolic isometry has strictly positive translation length.

In a geodesic space, a \defbold{labelled axis} of $g \in \Isom(X)$ is a geodesic $\gamma: \Rb \rightarrow X$ for which there exists $k > 0$ such that $g\gamma(t) = \gamma(t+k)$ for all $t \in \Rb$.  We see that in fact $k = |g|$, since $\frac{d(g^nx,x)}{n} = k$ for every $n > 0$ and every $x$ in the image of $\gamma$.  Thus the image of $\gamma$, called an \defbold{axis} of $g$, is contained in $\Min(g)$.  An isometry admitting an axis is called \defbold{axial}; clearly it is then also QI-hyperbolic.  

Proper $\delta$-hyperbolic spaces can admit isometries that are QI-hyperbolic but not axial.  However, given Lemma~\ref{lem:close_geodesic}, we see that if $X$ is a proper $\delta$-hyperbolic space and $g$ is a QI-hyperbolic isometry, then there is a geodesic line $\gamma$ from $\xi_-(g)$ to $\xi_+(g)$, and then any such $\gamma$ is a `quasi-axis', in that orbits of $g$ stay within a bounded distance of $\gamma(\Rb)$.

Let $\lambda = \inf\{d(gx,x) \mid x \in X\}$.  If $x \in X$ is such that $d(gx,x) = \lambda > 0$, we can construct an axis as follows: Take a geodesic $\sigma: [0,\lambda]$ with $\sigma(0) = x$ and $\sigma(\lambda) = gx$.  We then define $\gamma: \Rb \rightarrow X$ by setting, for $n \in \Zb$ and $0 \le r < 1$,
\[
\gamma(\lambda(n+r)) = g^n\sigma(\lambda r).
\]
In other words, we glue together the $\grp{g}$-translates of $\sigma$.  We see that $\lambda = |g|$ and $\gamma(\Rb)$ is contained in $\Min(g)$; the fact that $d(gx,x) \ge \lambda$ for all $x \in X$ ensures that $\gamma$ is a geodesic.  We thus arrive at the following description of $\Min(g)$.

\begin{lem}\label{lem:min_set}
Let $X$ be a geodesic space and let $g \in \Isom(X)$.
\begin{enumerate}[(i)]
\item If $|g| = 0$ then $\Min(g) = X^g$.
\item If $|g| > 0$ then $\Min(g)$ is the union of all axes of $g$.
\item The set $\Min(g)$ is empty if and only if the infimum $\inf\{d(gx,x)\mid x\in X\}$ is not realized by any $x \in X$.
\end{enumerate}
\end{lem}

For complete CAT(0) spaces, we note the following basic property of bounded sets.

\begin{lem}[{See \cite[Proposition~II.2.7 and Corollary~II.2.8]{BH}}]\label{lem:centre}
Let $X$ be a complete CAT(0) space and let $Y$ be a bounded subset of $X$.  Then there is a unique point $c = c_Y \in X$, called the \defbold{centre} of $Y$, that achieves the minimum value of $f_Y(c) := \sup \{d(c,y) \mid y \in Y\}$.  In particular, if $G \le \Isom(X)$ stabilizes $Y$, then $G$ fixes $c_Y$.
\end{lem}

An isometry $g$ of a complete CAT(0) space $X$ is called \defbold{parabolic} if $\Min(g)$ is empty.  Thus every isometry of $X$ either fixes a point, translates along an axis, or is parabolic, and the three cases are mutually exclusive. 

We say a boundary point $\xi$ is an \defbold{attracting point} for the isometry $g$, and write $\xi_+(g) = \xi$, if for all $x \in X$, the sequence $(g^nx)_{n \ge 0}$ converges to $\xi$ (for the notion of convergence appropriate to the space).  Note that the attracting point is unique if it exists.  If $\xi_+(g\inv)$ exists we also call it the \defbold{repelling point} of $g$ and define $\xi_-(g):= \xi_+(g\inv)$.  If $g$ is bounded, then clearly it has no attracting or repelling point.  On the other hand, if $g$ is QI-hyperbolic, then it is clear that both $\xi_+(g)$ and $\xi_-(g)$ exist and they are distinct.

If $X$ is a $\delta$-hyperbolic geodesic space, conventionally isometries are divided into three classes: those that are bounded, those that are QI-hyperbolic, and \defbold{parabolic} isometries, which are neither bounded nor QI-hyperbolic.  Parabolic isometries of $\delta$-hyperbolic spaces can equivalently be characterized as those unbounded isometries $g$ such that $(g^nx)_{n \ge 0}$ and $(g^{-n}x)_{n \ge 0}$ converge to the same point on the boundary: see \cite[Chapitre 9]{CDP}.

We will refer to axial isometries of complete CAT(0) spaces and QI-hyperbolic isometries of $\delta$-hyperbolic spaces collectively as \defbold{hyperbolic}.  To summarize, we have taken definitions such that every isometry of an NPC space $X$ is exactly one of: bounded, hyperbolic or parabolic.

\section{The stabilizers of the limit points of a hyperbolic isometry}

For this rest of the article we will generally assume that $X$ is an NPC space and that $G$ is a locally compact group that acts properly and continuously by isometries of $X$.  Let us note first that in this context, bounded elements are trivial from the perspective of scale theory.

\begin{lem}\label{lem:elliptic_unimodular}
Let $X$ be a metric space, let $G$ be a locally compact group acting continuously and properly by isometries and let $g \in G$ have bounded action on $X$.  Then $\para_G(g) = G$ and $\Delta_G(g)=1$.  If $G$ is totally disconnected, then $s_G(g)=s_G(g\inv)=1$.
\end{lem}

\begin{proof}
Let $\mu$ be a right-invariant Haar measure on $G$.  Let $x \in X$, and for $r > 0$ let $O_r$ be the set of all $h \in G$ such that $d(hx,x) \le r$.  Then $O_r$ is a compact identity neighbourhood, so $0 < \mu(O_r) < \infty$.  Since $g$ is bounded, there is some $c$ such that $d(g^nx,x) \le c$ for all $n \in \Zb$; it follows that for all $h \in O_r$ and $n \in \Zb$, we have
\[
d(g^nhg^{-n}x,x) = d(hg^{-n}x,g^{-n}x) \le d(hx,x) + 2d(g^{-n}x,x) \le r+2c.
\]
In particular, we have a bound on $\mu(g^nO_rg^{-n})$ over all $n \in \Zb$, namely $\mu(g^nO_rg^{-n}) \le \mu(O_{r+2c})$.  Thus $\Delta_G(g)=1$.  We also see that $O_r \subseteq \para_G(g)$; since $r$ can be made arbitrarily large, we deduce that $\para_G(g) = G$.  Similarly, $\para_G(g\inv)=G$.

If $G$ is totally disconnected, since $\para_G(g)$ and $\para_G(g\inv)$ are both open, we deduce from Lemma~\ref{lem:BW_scale_one} that $s_G(g)=s_G(g\inv)=1$.
\end{proof}

\subsection{The parabolic group}\label{sec:parabolic}

For NPC spaces, we have the following description of the stabilizer of the attracting point of a hyperbolic isometry.

\begin{prop}\label{prop:fixed_end_criterion}
Let $X$ be an NPC space, let $x \in X$ and let $g$ and $h$ be isometries of $X$, such that $g$ is hyperbolic.  Then $h$ stabilizes $\xi = \xi_+(g)$ if and only if the sequence $r_n = d(hg^nx,g^nx)$ is bounded over all $n \ge 0$.  Moreover, if $h$ stabilizes $\xi$ then there exists $n_0$ such that for all $n \ge n_0$ we have
\[
d(hg^nx,g^nx) \le d(hx,x) + c_{g,x},
\]
where $c_{g,x}$ is a constant independent of $h$, and where for fixed $(g,x)$, $n_0$ can be bounded by a function of $d(hx,x)$ of linear growth.
\end{prop}

\begin{proof}
Write $x_n = g^nx$.  If the sequence $d(hx_n,x_n)$ is bounded over all $n \ge 0$, we see that both of the sequences $(x_n)_{n \ge 0}$ and $(hx_n)_{n \ge 0}$ converge to $\xi$, so $h\xi = \xi$.

Conversely, suppose $h$ stabilizes $\xi$.

If $g$ is axial and $X$ is a complete CAT(0) space, take $\gamma$ to be a labelled axis for $g$.  Then by Lemma~\ref{lem:decreasing_distance} and the triangle inequality, for all $t \ge 0$ we have 
\[
d(h\gamma(t),\gamma(t)) \le d(h\gamma(0),\gamma(0)) \le d(hx,x)+2d(x,\gamma(0)).
\]
At the same time, we have
\[
d(\gamma(t_n),x_n) = d(g^n\gamma(0),g^nx_n) = d(x,\gamma(0)), \text{ where } t_n = n|g|.
\]
It follows that for all $n$ we have 
\[
d(hx_n,x_n) \le d(hx_n,h\gamma(t_n)) + d(h\gamma(t_n),\gamma(t_n)) +  d(\gamma(t_n),x_n) \le d(hx,x) + 4d(x,\gamma(0)),
\]
so $d(hx_n,x_n)$ is bounded as claimed (with $n_0=0$).

If instead $X$ is a proper $\delta$-hyperbolic space and $g$ is QI-hyperbolic, let $\gamma$ be a geodesic ray from $x$ to $\xi$.  Then by Lemma~\ref{lem:thin_omega_triangle}, for all $t > 24\delta+d(hx,x)$ we have 
\[
d(h\gamma(t),\gamma(t)) \le d(hx,x) + 48\delta.
\]
At the same time, the sequence $(x_n)$ converges to $\xi$ and is contained in a quasi-geodesic, so by Lemma~\ref{lem:close_geodesic} there is a constant $c_1$ (depending on $(g,x)$) and a sequence of real numbers $t_n$ with linear growth rate (that is, such that $n \mapsto t_n$ is a quasi-geodesic in $\Rb$) such that $d(x_n,\gamma(t_n)) \le c_1$.  Take $n_0$ such that $t_n > 24\delta+d(hx,x)$ for all $n \ge n_0$; note that $n_0$ can be taken to depend linearly on $d(hx,x)$ as $d(hx,x) \rightarrow \infty$.  It follows that for all $n \ge n_0$ we have 
\[
d(hx_n,x_n) \le d(hx,x) + 48\delta + 2c_1,
\]
so again, $d(hx_n,x_n)$ is bounded as claimed.
\end{proof}

We can now prove the theorem from the introduction about the parabolic group.

\begin{proof}[Proof of Theorem~\ref{thm:para_stable}]
Let $v_0 \in X$ and let $v_n = g^{-n}v_0$; thus $(v_n)$ converges to $\xi$.  Let $u \in G$ and write $k_n = g^nug^{-n}$; thus
\[
uv_n = g^{-n}g^nug^{-n}v_0 = g^{-n}k_nv_0,
\]
and hence $d(uv_n,v_n) = d(k_nv_0,v_0)$.

Suppose $u \in \para_G(g)$; then sequence $k_n$ is confined to a compact subset of $G$, and so by continuity the set $\{k_nv_0\}$ is bounded.  Since $d(uv_n,v_n) = d(k_nv_0,v_0)$, it follows that $d(uv_n,v_n)$ is bounded independently of $n$, and consequently the sequence $(uv_n)_{n \ge 0}$ converges to $\xi$.  In particular, $u\xi = \xi$, so $\para_G(g)$ stabilizes $\xi$.

Now suppose that $g$ is hyperbolic and the action of $G$ is proper, and let $u \in G_{\xi}$.  Then $d(k_nv_0,v_0) = d(uv_n,v_n)$ is bounded over all $n \ge 0$ by Proposition~\ref{prop:fixed_end_criterion}.  By properness of the action it follows that $k_n$ is confined to a compact subset, and hence $u \in \para_G(g)$.
\end{proof}

By standard scale theory results, we thus gain insight into the scale of a hyperbolic $g$ by considering the stabilizer of its attracting and repelling points.

\begin{cor}\label{cor:parabolic_scale}
Let $X$ be an NPC space and let $G$ be a \tdlc group acting continuously by isometries. Suppose $g \in G$ has an attracting point $\xi_+$ and repelling point $\xi_-$.
\begin{enumerate}[(i)]
\item We have $s_G(g) = s_{G_{\xi_+}}(g)$.
\item If $s_G(g)=1$ then $G_{\xi_-}$ is open.
\item If $G$ acts properly on $X$ and $g$ is hyperbolic, then $s_G(g) = \Delta_{G_{\xi_+}}(g)$.
\item If $G$ acts properly on $X$, $g$ is hyperbolic and $G_{\xi_-}$ is open, then $s_G(g)=1$.
\end{enumerate}
\end{cor}

\begin{proof}
By Theorem~\ref{thm:para_stable} we have $G_{\xi_-} \ge \para_G(g)$ and $G_{\xi_+} \ge \para_G(g\inv)$.  Part (i) then follows from Lemma~\ref{lem:scale_subgroup} and part (ii) from  Lemma~\ref{lem:BW_scale_one}.

Now suppose that $G$ acts properly on $X$ and $g$ is hyperbolic; then by Theorem~\ref{thm:para_stable} we have $G_{\xi_-} = \para_G(g)$ and $G_{\xi_+} = \para_G(g\inv)$.  Part (iii) now follows from Corollary~\ref{cor:scale_subgroup:modular} and part (iv) from Lemma~\ref{lem:BW_scale_one}.
\end{proof}

We note the following special case of Theorem~\ref{thm:para_stable}.

\begin{cor}\label{cor:compact_fix}
Let $X$ be an NPC space, let $G$ be a locally compact group acting continuously by isometries and let $K$ be a compact normal subgroup of $G$.  Then $K$ stabilizes every point at infinity that occurs as the attracting point of some $g \in G$.
\end{cor}

\begin{proof}
Since $K$ is compact and normal, we have $K \subseteq \para_G(g)$ for every $g \in G$.  In particular, if $\xi = \xi_+(g)$ for some $g \in G$, and hence we have $K \le \para_G(g\inv) \le G_{\xi}$.
\end{proof}

\subsection{Orbits of hyperbolic limit points}

We note the following conditions ensuring the stability of dynamics of hyperbolic elements with respect to the topology of $G$.

\begin{lem}\label{lem:stable_dynamics}
Let $X$ be an NPC space, let $G$ be a \tdlc group acting continuously by isometries and let $g \in G$ be hyperbolic, with attracting point $\xi_+$ and repelling point $\xi_-$.  Let $U$ be a tidy above subgroup for $g$, let $V = U \cap gUg\inv$ and let $h \in VgV$.  Then $|h| = |g|$ and there exists $r \in V$ such that $\xi_+(h) = r\xi_+$ and $\xi_-(h) = r\xi_-$.  Moreover, if either $X$ is $\delta$-hyperbolic or $U$ has a fixed point in $\Min(g)$, then $h$ is hyperbolic.
\end{lem}

\begin{proof}
Since the desired properties of $h$ are clearly stable under conjugation in $V$, we may assume $h = ug$ for some $u \in V$.  By Lemma~\ref{lem:tidy_stability}, there is then some $r \in V$ such that $h^n \in rg^n Vr\inv$ for all $n \in \Zb$.  By continuity, $d(vr\inv x,x)$ is bounded over $v \in V$, so $d(rg^nvr\inv x,rg^nr\inv x)$ is also bounded over $v \in V$ and $n \in \Zb$.  In particular, the sequence $(h^nx)_{n \ge 0}$ is asymptotic to $(rg^nr\inv x)_{n \ge 0}$, and hence $h$ has attracting point $r\xi$.  Similarly, $h$ has repelling point $r\xi$.  Since $d(h^nx,rg^nr\inv x)$ is bounded over $n \in \Zb$, we see that $|h| = |g|$.

In the case that $X$ is $\delta$-hyperbolic, the fact that $\xi_+(h) \neq \xi_-(h)$ immediately ensures that $h$ is hyperbolic.  If instead $X$ is a complete CAT(0) space and $U$ has a fixed point $x$ in $\Min(g)$, then $hx = ugx = gx$, so $x \in \Min(h)$.  In particular, $\Min(h)$ is nonempty; since $h$ is unbounded we deduce from Lemma~\ref{lem:min_set} that $h$ is hyperbolic.
\end{proof}

\begin{cor}\label{cor:countable_hyperbolic}
Let $X$ be an NPC space, and let $G$ be a $\sigma$-compact \tdlc group acting smoothly by isometries.  Then $G$ has at most countably many orbits on $\partial^{\hb}_X G$ and in particular, at most countably many points in $\partial^{\hb}_X G$ have open stabilizer in $G$.
\end{cor}

\begin{proof}
For each hyperbolic $g \in G$, Lemma~\ref{lem:stable_dynamics} ensures there is a neighbourhood of $g$ consisting entirely of hyperbolic elements whose limit points are in the same $G$-orbit as those of $g$.  This leads to a partition of the hyperbolic elements of $G$ into open sets.  Since $G$ is $\sigma$-compact, any collection of disjoint open sets in $G$ is countable, and we deduce that $G$ has at most countably many orbits on $\partial^{\hb}_X G$.  Since $G$ is $\sigma$-compact, any orbit of locally invariant points in $\partial^{\hb}_X G$ is itself countable; thus the set of locally invariant points in $\partial^{\hb}_X G$ under the action of $G$ is countable.
\end{proof}

\subsection{Rigidity of coaxial pairs}\label{sec:uniscalar_hyperbolic}

As before we let $X$ be an NPC space, and let $G$ be a group acting on $X$ by isometries.

Given $\xi_1,\xi_2 \in \partial X$, say that the pair $(\xi_1,\xi_2)$ is \defbold{coaxial} if there exists a hyperbolic $g \in G$ such that $\xi_+(g) = \xi_1$ and $\xi_-(g) = \xi_2$.  Given $\xi_1,\xi_2 \in \partial^{\hb}_X G$, recall that we write $\xi_1 \rightarrow \xi_2$ if for all $g \in G$ hyperbolic we have $\xi_+(g) = \xi_1 \Rightarrow \xi_-(g) = \xi_2$.

%We now define some relations on $\partial^{\hb}_X G$.
%\begin{itemize}
%\item We write $\xi_1 \rightarrow \xi_2$ if $\xi_1 \in \partial^{\hb}_X G$ and for all $g \in G$ hyperbolic such that $\xi_+(g) = \xi_1$, then $\xi_-(g) = \xi_2$.
%%\item If $\xi_1 \rightarrow \xi_2$ and $\xi_2 \rightarrow \xi_1$, we write $\xi_1 \leftrightarrow \xi_2$.
%%\item If $(\xi_1,\xi_2)$ is coaxial, but neither $\xi_1 \rightarrow \xi_2$ nor $\xi_2 \rightarrow \xi_1$, we write $\xi_1 \simgt \xi_2$.
%\end{itemize}

In this subsection we will obtain equivalent descriptions of when $\xi_1 \rightarrow \xi_2$ for a proper continuous action of the \tdlc group $G$ acting properly and continuously in terms of the scale function, which points in $\partial^{\hb}_X G$ have open stabilizer, and which stabilizers contain each other.  This will lead to the proof of Theorem~\ref{thm:uniscalar_hyperbolic}.

The first equivalence we obtain does not involve the group topology.

\begin{lem}\label{lem:fixed_points_hyperbolic}
Let $X$ be an NPC space and let $g$ be a hyperbolic isometry.  Let $\xi_1 = \xi_+(g)$ and take $\xi_2 \in \partial X$ opposite $\xi_1$.  Then $g$ stabilizes $\xi_2$ if and only if $\xi_2 = \xi_-(g)$.
\end{lem}

\begin{proof}
If $X$ is $\delta$-hyperbolic, the conclusion follows immediately from the fact that $g$ has only two fixed points in $\partial X$ (e.g. \cite[Chapitre 10 Proposition 6.6]{CDP}), so from now on, we may assume that $X$ is a complete CAT(0) space.  Certainly $g$ stabilizes $\xi_-(g)$, so let us suppose $\xi_-(g) \not\in \{\xi_1,\xi_2\}$ and obtain a contradiction.

Let $Y$ be the union of all geodesic lines from $\xi_1$ to $\xi_2$; by assumption $Y$ is nonempty.  By the product decomposition theorem (see \cite[II.2.14]{BH}), $Y$ is a closed convex space that decomposes as $Y^* \times \Rb$, where the geodesic lines from $\xi_1$ to $\xi_2$ are the fibres $\{y^*\} \times \Rb$.  Since $g$ stabilizes $\xi_1$ and $\xi_2$, it acts on $Y$ and respects the product decomposition.  Let $x \in Y$ and let $\pi$ be the natural projection from $Y$ to $Y^*$.  Since $\xi_+(g) = \xi_1$, we see that each $g$-orbit stays within a bounded distance of $\{\pi(x)\} \times \Rb$; in particular, $\{\pi(g^{n}x)_{n \ge 0}\}$ has finite diameter.  Since $g$ induces an isometry of $Y^*$, it follows that $\{\pi(g^{-n}x)_{n \ge 0}\}$ also has finite diameter.  But then $\xi_-(g) \in \{\xi_1,\xi_2\}$ and we have our contradiction.
\end{proof}

Note that if $X$ is a complete CAT(0) space, then $g$ can stabilize points at infinity that are not opposite $\xi_+(g)$: for example, any translation of $\Rb^2$ fixes $\partial \Rb^2$.

\begin{cor}\label{cor:fixed_points_hyperbolic}
Let $X$ be an NPC space, let $G$ be a group acting on $X$ by isometries.  Let $\xi_1 \in \partial^{\hb}_X G$ and take $\xi_2 \in \partial X$ opposite $\xi_1$.  Then $\xi_1 \rightarrow \xi_2$ if and only if $G_{\xi_1} \le G_{\xi_2}$.
\end{cor}

\begin{proof}
If $\xi_1 \rightarrow \xi_2$ and $g$ in $G$ stabilizes $\xi_1$, then we have $\xi_1 \rightarrow g\xi_2$; the uniqueness of $\xi_2$ then ensures $g\xi_2 = \xi_2$.  Thus $G_{\xi_1} \le G_{\xi_2}$.  The converse follows from Lemma~\ref{lem:fixed_points_hyperbolic}.
\end{proof}

We now return to the setting of locally compact groups, where we can incorporate the modular function into the characterization of the relation $\rightarrow$.

\begin{prop}\label{prop:modular}
Let $X$ be an NPC space and let $G$ be a locally compact group acting continuously and properly by isometries on $X$.  Let $g \in G$ be hyperbolic, let $\xi_2 \in \partial X$ be opposite $\xi_+(g)$, and let $H$ be the stabilizer of $\xi_+(g)$ in $G$.  Then the following are equivalent:
\begin{enumerate}[(i)]
\item $H$ stabilizes $\xi_2$;
\item $\xi_1 \rightarrow \xi_2$;
\item $\Delta_H(g)=1$ and $\xi_-(g) = \xi_2$.
\end{enumerate}
\end{prop}

\begin{proof}
Write $\xi_- = \xi_-(g)$, and choose a Haar measure $\mu$ on $H$ and a point $x \in X$.  Given $r \ge 0$, let $O_r$ be the set of $h \in H$ such that $d(hx,x) \le r$; note that $O_r$ is a compact identity neighbourhood for all $r>0$.

Cases (i) and (ii) are equivalent by Corollary~\ref{cor:fixed_points_hyperbolic}.

Suppose (ii) holds; clearly $\xi_- = \xi_2$.  By (i) and Proposition~\ref{prop:fixed_end_criterion}, there exists $n_0$ and a constant $c$ such that for all $|n| \ge n_0$ and $h \in O_1$, we have $d(hg^nx,g^nx) \le c$.  In other words, for $|n| \ge n_0$, we have $g^nO_1g^{-n} \subseteq O_c$ and hence $\mu(g^nO_1g^{-n}) \le \mu(O_c) < \infty$.  Thus the $\grp{g}$-conjugates of $O_1$ have bounded measure, showing that $\Delta_H(g)=1$.  Thus (ii) implies (iii).

We now suppose (iii) holds and aim to prove (i).  Suppose for a contradiction that $H$ does not stabilize $\xi_-$.  Let $x \in X$ and $r > 0$.  Then $O_r$ is a compact open identity neighbourhood in $G$, so $0 < \mu(O_r) < \infty$, and we have $H = \bigcup_{r > 0}O_r$.  Since $\xi_-$ is not fixed by $H$, by choosing $r$ large enough we may ensure that there is $h \in O_{r/2}$ such that $h\xi_- \neq \xi_-$.

By Proposition~\ref{prop:fixed_end_criterion} there is a neighbourhood $E$ of $h$ and a natural number $n_0$ such that for all $k \in E$ and $n \ge n_0$ we have $d(kg^nx,g^nx) \le r+c$, where $c$ is independent of $k$.  In other words, we have $E \subseteq K_{n_0}$, where $K_n$ denotes the compact set $\bigcap_{m \ge n}g^mO_{r+c}g^{-m}$; note that by construction, $K_n \subseteq K_{n+1} = gK_ng\inv$.

At the same time, since $h$ does not stabilize $\xi_-$ we have $h \not\in \para_H(g)$, and hence $d(kg^{-n}x,g^{-n}x)$ is unbounded for $n \ge 0$.  Thus there exists $n$ such that
\[
d(g^nhg^{-n}x,x) = d(hg^{-n}x,g^{-n}x) \ge r + 2c.
\]
There is then an open neighbourhood $E'$ of $h$ in $E$ such that
\[
\forall k \in E': d(g^nkg^{-n}x,x) > r + c.
\]
In particular, we see that $E'$ is disjoint from $g^{-n}O_{r+c}g^n$, hence from $K_{-n}$, while being contained in $K_n$.  Since $K_n$ contains $K_{-n}$, we have
\[
\mu(K_n) \ge \mu(K_{-n}) + \mu(E') > \mu(K_{-n});
\]
since $K_n = g^{2n}K_{-n}g^{-2n}$, we deduce that $\Delta_H(g^{2n}) > 1$ and hence $\Delta_H(g) > 1$, giving the required contradiction.  We have now shown that (i)--(iii) are equivalent.
\end{proof}

We have almost finished the proof of the main theorem from the introduction on hyperbolic elements of scale $1$.

\begin{proof}[Proof of Theorem~\ref{thm:uniscalar_hyperbolic}]
(i)--(iv) are equivalent by Corollary~\ref{cor:parabolic_scale}.  By Proposition~\ref{prop:modular} the statements (iii), (vi) and (vii) equivalent.  Clearly (vi) implies (v).  Finally, if (v) holds then (ii) follows by applying Corollary~\ref{cor:parabolic_scale}(iv) to $g$ as an element of $G_{\xi_+}$.  This completes the proof that (i)--(vii) are equivalent.
\end{proof}

\subsection{Axis trees}\label{sec:axis_tree}

We now prove the geometric interpretation of the scale for axial isometries from the introduction.  Recall the notation introduced in Definition~\ref{defn:axis_tree}.

\begin{proof}[Proof of Theorem~\ref{thm:axis_tree}]
Recall that $L_{G,\rho}$ consists of all geodesic lines that contain $\rho([t,+\infty))$ for some $t \in \Rb$, where $\rho: \Rb \rightarrow X$ is a specified geodesic line.  Since $X$ is uniquely geodesic, in fact the intersection of any two elements of $L_{G,\rho}$ must be a ray representing $\xi = \rho(+\infty)$.  The union $T$ of $L_{G,\rho}$ is therefore an $\Rb$-tree, in other words, between any two points in $T$ there is a unique non-backtracking path.  In particular, $T$ is itself a uniquely geodesic space.

We now consider the group $H$.  From the construction of $T$, we see that $H$ is the stabilizer in $G$ of the equivalence class of $\rho_+$, where $g\rho_+$ is equivalent to $\rho_+$ if the intersection of the images of $\rho_+$ and $g\rho_+$ is a ray representing $\xi$.  From this description it is clear that $S \subseteq H$; clearly also $g \in S$.  Now consider $h \in H$ in relation to $\rho$.  We see that $h\rho$ is another axis of translation of $G$ ending at $\xi$.  The fact that $T$ is an $\Rb$-tree then forces this axis to contain $\rho([t_0,+\infty))$ for some $t_0 \in \Rb$; given that $h$ is an isometry stabilizing $\xi$, the only possibility is that there is some (necessarily unique) $\beta(h) \in \Rb$ such that
\[
\rho([t_0,+\infty)) = h\rho([t_0+\beta(h),+\infty)) \text{ and } \forall t \ge t_0+\beta(h): h\rho(t) = \rho(t-\beta(h)).
\]
Thus we have a map $\beta: H \rightarrow \Rb$; given how elements of $h$ act on $\rho(t)$ for $t \gg 0$, it is easy to see that $\beta$ is a continuous homomorphism (indeed $\beta$ is a Busemann character on $(T,d_T)$, see for instance \cite[\S3.C]{CM-CAT}).  If $\beta(h) = 0$, then $h$ fixes a point, so $h$ is bounded.  If $\beta(h) \neq 0$, we see that the $\grp{h}$-translates of $R = \rho([t'_0,+\infty))$ are totally ordered by inclusion for all $t'_0 \ge t_0+|\beta(h)|$, and that in fact $\bigcup_{n \in \Zb}h^nR$ forms an axis for $h$, so $h$ is hyperbolic; it is then clear that $\xi$ is an attracting end of $h$ if and only if $\beta(h)<0$.  We have now proved (ii) except for showing the infimum $\lambda$ of the translation lengths is strictly positive.

Since $\rho_+$ is the unique geodesic ray starting at $\rho(0)$, we see that the group $A = G_{\rho(0)} \cap G_\xi$ fixes $\rho([0,+\infty))$.  From there it is easy to see that $A$ stabilizes $T$, so $A \le H$.  Since $G_{\rho(0)}$ is open in $G$, we deduce that $H$ is open in $G_{\xi}$; in particular, $H$ certainly contains $\con_{G_\xi}(h^{-1})$ for all $h \in H$.  It follows from Corollary~\ref{cor:scale_subgroup:modular} that $s_H(h) = s_{G_\xi}(h)$ for all $h \in H$.  In turn, for $h \in S$, we know by Corollary~\ref{cor:parabolic_scale} that $s_{G_\xi}(h) = s_G(h)$.  This completes the proof of (i).

Suppose $s_G(g)=1$.  Then by Theorem~\ref{thm:uniscalar_hyperbolic}, we see that every element of $S$ has repelling end $\xi_-(g)$.  Since $T$ is an $\Rb$-tree formed as a union of axes of elements of $S$, we deduce that $T$ is a line, so $\sigma_{T,x}(m)=1$ for all $m \ge 0$ and $x \in T$.  The fact that $H$ acts properly now ensures that $H/K$ is cyclic, where $K$ is the kernel of the action of $H$ on $T$; since $K$ is compact, it is then clear that $H$ is uniscalar.  The remaining conclusions of the theorem are now clear, so we may assume instead that $s_G(g)>1$.

Using elements of $S$ and their inverses, we see that every $H$-orbit on $T$ intersects the compact line segment $Y = [\rho(0),\rho(|g|)]$.  By hypothesis the point stabilizer $G_{\rho(0)}$ is open, and since $H$ also stabilizes $\xi$, we see that $H_{\rho(0)}$ fixes $Y$.  The fact that $H$ acts properly then ensures that for all $r>0$ the ball $B_r$ of radius $r$ around $\rho(0)$ can only intersect finitely many $H$-translates of $Y$, so the intersection of $T$ with $B_r$ is contained in a finite union of $H$-translates of $Y$.

Let $V$ be the set of points $v$ of $T$ such that $T \setminus \{v\}$ has more than two connected components.  The assumption that $s_G(g)>1$ ensures that also $s_H(g)>1$, by part (i); thus $T$ is not a line, which ensures that $V$ is not empty.  On the other hand, the covering of $T$ by $H$-translates of $Y$ ensures that $V$ has finite intersection with $B_r$, so $V$ is discrete.  In particular, $V \cap Y$ is finite, so $H$ has finitely many orbits on $V$.  We see that the connected components of $T \setminus V$ are open line segments, which admit $H$-translates inside $Y$ and hence have length at most $|g|$.  Replacing $g$ with some $h \in S$ does not materially change the construction of $T$, so in fact the edge lengths are all bounded above by $\lambda$, and hence $\lambda>0$.  We can then build a combinatorial graph structure with vertex set $V$ and an edge between distinct vertices $v$ and $w$ of length $d_T(v,w)$ if there is a $T$-geodesic from $v$ to $w$ that does not pass through any other point in $V$.  Note that such a geodesic will be contained in some element of $L_{G,\rho}$, so it is actually an $X$-geodesic. We see that $H$ naturally acts continuously by isometries on the resulting graph; moreover, each vertex has less than $|B_{|g|} \cap V|$ neighbours, so the graph is locally finite and the action of $H$ on $T$ is proper and cocompact.  In particular, the minimum translation length $\lambda$ is achieved by some $h \in S$; without loss of generality, $|g|=\lambda$.  We then see that the line segment $Y$ witnesses each $H$-orbit of edges of the combinatorial tree exactly once, so $\lambda$ is the sum of the edge lengths.  This completes the proof of (ii), (iii) and (iv).

All that remains is to prove (v).  Take some $h \in H$ and consider $\beta(h)$.  If $\beta(h)=0$ then $h$ is bounded and hence $s_H(h)=1$, since $H$ acts properly.  If $\beta(h) > 0$, then $h$ has repelling end $\xi$, so $s_{G_{\xi}}(h)=1$, and hence $s_H(h)=1$.  Given (ii), from now on we may assume $\beta(h) < 0$, in other words $h \in S$.  Write $R_t = \rho([t,+\infty))$; we take $t$ such that the $\grp{h}$-translates of $R_t$ are totally ordered by inclusion (by our previous argument, this will be the case for $t \gg 0$).  Let $W_t = H_{\rho(t)}$.  Then $W_t$ is compact and open in $H$, and fixes $R_t$; moreover, $hR_t$ is contained in $R_t$, so $hW_th\inv \ge W$, and hence $W_t$ is tidy for $h$ as an element of $H$.  We then obtain the scale as $s_H(h) = |hW_th\inv:W_t|$, or equivalently, $s_H(h)$ is the number of points in $O_t$, where $O_t$ is the $H_{\rho(t+|h|)}$-orbit of $\rho(t)$.  For all $x \in O_t$ we see that $x \le_T \rho(t+|h|)$, with $d(\rho(t+|h|),x) = |h|$.  On the other hand, given $x,y \in T$ such that $x,y \le_T \rho(t+|h|)$ and $d(\rho(t+|h|),x) = |h|$, we see there exist $s_x,s_y \in S$ such that both $s_xx,s_yy \in R_t$, and given the position of $x$ and $y$ relative to $\rho(t)$, we can in fact arrange to have $s_xx = s_yy$.  Given $h' \in H$ such that $h'x = y$, we see that $\lambda_{h'} = 0$; the rays from $x$ and $y$ going towards $\xi$ both contain $\rho(t+|h|)$, from which we deduce that $h'$ fixes $\rho(t+|h|)$.  From this argument we deduce that 
\[
O_t = \{y \in X \mid y \le_T \rho(t+|h|), d_T(\rho(t+|h|),y) = m\},
\]
so 
\[
s_H(h) = O_t = \sigma_{T,\rho(t+|h|)}(|h|).
\]
Now $s_H(h)$ does not depend on the choice of $t \gg 0$; by varying $t$, we can choose for the point $x_0 = \rho(t+|h|)$ to lie in any $H$-orbit, and then it is clear that $\sigma_{T,x_0} = \sigma_{T,h'x_0}$ for all $h' \in H$.  Thus $\sigma_{T,x}(|h|) = s_H(h)$ for all $x \in T$, which completes the proof.
\end{proof}

\section{Tidy subgroups}\label{sec:tidy}

We are ready to prove Theorem~\ref{thm:intro_tidy} from the introduction; let us first recall the relevant conditions on the compact open subgroup $U$.

\begin{defn}
Let $X$ be an NPC space, let $G$ be a \tdlc group acting properly and continuously by isometries and let $g \in G$ be hyperbolic, with attracting and repelling ends $\xi_+$ and $\xi_-$ respectively.  Let $U$ be a compact open subgroup of $G$.

Say that $U$ is \defbold{geometrically tidy above for $g$} ($\mathrm{GTA}(g)$) if $U = U_{\xi_+}U_{\xi_-}$.

Say that $U$ is \defbold{geometrically positively aligned with $g$} ($\GTp(g)$) if $g\inv U_{\xi_+}g \le U$.

Say that $U$ is \defbold{geometrically negatively aligned with $g$} ($\GTm(g)$) if $g U_{\xi_-}g\inv \le U$.  (Note that $\GTm(g) = \GTp(g\inv)$.)
\end{defn}

\begin{proof}[Proof of Theorem~\ref{thm:intro_tidy}]
Write $U_+ = \bigcap_{n \ge 0} g^nUg^{-n}$ and $U_- = \bigcap_{n \ge 0} g^{-n}Ug^n$.

Suppose $U$ is minimizing for $g$.  Then by Theorem~\ref{thm:Willis} have $U = U_+U_-$.  Note that $g\inv U_+g, gU_-g\inv \le U$ by construction.  In turn, by Lemma~\ref{lem:parabolic_tidy}, we have $U \cap \para_G(g\inv) = U_+$ and $U \cap \para_G(g) = U_-$.  Given Theorem~\ref{thm:para_stable}, we therefore have $U_+ = U_{\xi_+}$ and $U_- = U_{\xi_-}$.  We now see that all the conditions $\GTA(g),\GTA(g),\GTm(g)$ are satisfied.

Conversely, suppose that $U$ satisfies $\GTA(g)$, $\GTp(g)$, and $\GTm(g)$.  We see that
\[
U_{\xi_+} \le gUg\inv \cap G_{\xi_+} = gU_{\xi_+}g\inv,
\]
and similarly, $gU_{\xi_-}g\inv \le U_{\xi_-}$.  In particular, $U \cap gUg\inv$ contains $U_{\xi_+} gU_{\xi_-}g\inv$, so
\[
|gUg\inv:U \cap gUg\inv| \le \frac{\mu(gUg\inv)}{\mu(U_{\xi_+} gU_{\xi_-}g\inv)} \le |gU_{\xi_+}g\inv:U_{\xi_+}|,
\]
where $\mu$ is any right-invariant Haar measure for $G$.  In turn, we see that 
\[
|gU_{\xi_+}g\inv:U_{\xi_+}| = \Delta_{G_{\xi_+}}(g) = s_G(g),
\]
where the last equality is by Corollary~\ref{cor:parabolic_scale}.  Thus $U$ is minimizing for $g$.
\end{proof}

\begin{rem}\label{rem:geometrically_tidy}
An equivalent way of stating that $U$ satisfies $\mathrm{GTA}(g)$ is to say that $U$ acts transitively on $U\xi_+ \times U\xi_-$, where $U\xi$ denotes the orbit of $U$ on $\partial X$ containing $\xi$.

The condition $\GTp(g)$ can be expressed in a few equivalent ways (and similarly for $\GTm(g)$):
\begin{enumerate}[(a)]
\item $g\inv U_{\xi_+}g \le U$;
\item $g\inv U_{\xi_+}g \le U_{\xi_+}$;
\item $U_{\xi_+}$ is tidy for the action of $g$ on $G_{\xi_+}$;
\item We have $U_{\xi_+} \le U_+$.
\end{enumerate}
\end{rem}

Corollary~\ref{cor:intro_tidy_neighbourhood} is almost immediate; Corollary~\ref{cor:intro_tidy_axis} takes a little more explanation.

\begin{proof}[Proof of Corollary~\ref{cor:intro_tidy_neighbourhood}]
It is enough to consider the case that $U$ is compact.  By Theorem~\ref{thm:intro_tidy} the product $V_{\xi_+}V_{\xi_-}$ is a neighbourhood $V$ of the identity in $G$ for some $V \in \mc{COS}(G)$; since $U$ contains a finite index subgroup of $V$, we can cover $V = V_{\xi_+}V_{\xi_-}$ by finitely many double cosets $U_{\xi_+}hU_{\xi_-}$.  Thus $V$ is partitioned into finitely many compact sets (and hence finitely many open sets) of the form $U_{\xi_+}hU_{\xi_-} \cap V$.  In particular, $U_{\xi_+}U_{\xi_-}$ is a neighbourhood of the identity in $G$.
\end{proof}

\begin{proof}[Proof of Corollary~\ref{cor:intro_tidy_axis}]
Recall that we have assumed the image of $\gamma: \Rb \rightarrow X$ is an axis for $g$; write $\xi_+$ for the attracting end of $g$ and $\xi_-$ for the repelling end of $g$.  Without loss of generality, there is $a > 0$ such that $g\gamma(r) = \gamma(r+a)$ for all $r \in \Rb$, so that $\gamma(s) \rightarrow \xi_{\pm}$ as $s \rightarrow \pm \infty$.

Let $r \in \Rb$ and $t \ge 0$ and let $U(r,t) = G_{\gamma(r),\gamma(r+t)}$.  The condition that $G_{\gamma(0)}$ is open ensures that $G_{\gamma(-ma),\gamma(na)}$ is open for all $m,n \in \Nb$, from which we deduce that $U(r,t)$ is open; by properness of the action, $U(r,t)$ is a compact subgroup of $G$.  Moreover, we have $U(r,t)_{\xi_+} = \Fix_G(\gamma([r,+\infty)))$ and $U(r,t)_{\xi_-} = \Fix_G(\gamma((-\infty,r+t]))$; it is then clear that $gU(r,t)_{\xi_+}g\inv \ge U(r,t)_{\xi_+}$ and $g U(r,t)_{\xi_-}g\inv \le U(r,t)_{\xi_-}$, ensuring that $U(r,t)$ satisfies $\GTp$ and $\GTm$.

By Corollary~\ref{cor:intro_tidy_neighbourhood} the product $U(0,0)_{\xi_+}U(0,0)_{\xi_-}$ is a neighbourhood of the identity in $G$.  We have $\bigcap_{t \ge 0}U(0,t) \le U(0,0)_{\xi_+}$; since the sets $U(0,t)$ are all compact subgroups, there is thus $t_0 \ge 0$ such that 
\[
\forall t \ge t_0: \; U(0,t) \subseteq U(0,0)_{\xi_+}U(0,0)_{\xi_-}.
\]
We see that $U(0,0)_{\xi_+} = U(0,t)_{\xi_+} \le U(0,t)$ and $U(0,t)_{\xi_-} \le U(0,0)_{\xi_-}$, so in fact we obtain
\[
\forall t \ge t_0: \; U(0,t) = U(0,t)_{\xi_+}U(0,t)_{\xi_-};
\]
in other words, for all $t \ge t_0$ then $U(0,t)$ satisfies $\GTA$, and hence is tidy for $g$ by Theorem~\ref{thm:intro_tidy}.
\end{proof}

\begin{ex}\label{ex:small_tidy}
Let $n \ge 2$, let $F$ be the free product of $n^2$ copies of $\Zb/2\Zb$, with the copies indexed by $\Zb/n\Zb \times \Zb/n\Zb$, and let $T$ be the corresponding Cayley graph of $F$.  (We will write elements of $\Zb/n\Zb$ as integers, which should be understood to be read modulo $n$; we also identify $T$ with its geometric realization in order to consider it as an NPC space.)  Thus $T$ is a tree in which each vertex has $n^2$ neighbours, and the edges are labelled in such a way that the edges incident with each vertex are in bijection with $\Zb/n\Zb \times \Zb/n\Zb$.  Then for each $g \in \Isom(T)$ and $v \in VT$ there is a permutation $\sigma(g,v)$ of $\Zb/n\Zb \times \Zb/n\Zb$, the \defbold{local action of $g$ at $v$}, such that if $e$ is an edge incident with $v$ of colour $c$, then $ge$ has colour $\sigma(g,v)(c)$.  We let $A = \Sym(n) \wr C$, where the $j$-th copy $S_j$ of $\Sym(n)$ has natural action on $\Zb/n\Zb \times \{j\}$ and fixes $\Zb/n\Zb \times \{j'\}$ for $j' \neq j$, and where the cyclic group $C \cong \Zb/n\Zb$ on top sends $(i,j)$ to $(i,j+1)$ for all $i,j \in \Zb/n\Zb$.  Let $G_0 = U(A)$ be the universal group with local action $A$ (in the sense of Burger--Mozes \cite{BurgerMozes}).

We define a closed subgroup $G$ of $G_0$ by imposing the following additional condition on elements $g \in G$: for each pair of vertices $v,w$ connected by an edge of colour $(i,j)$, then $\sigma(g,v)^{-1}\sigma(g,w)$ is an element of $S_{(i,j)} := S_j$.  To check this defines a subgroup of $G_0$, we consider $\sigma(gh^{-1},v_t)$ for $g,h \in G$ and $v_0,v_1 \in VT$ connected by an edge of colour $(i,j)$.  Write $(i',j')$ for the colour of the edge $(h^{-1}v_0,h^{-1}v_1)$.
\begin{align*}
\sigma(gh^{-1},v_0)^{-1}\sigma(gh^{-1},v_1) &=  \left(\sigma(h,h^{-1}v_0)\sigma(g,h^{-1}v_0)^{-1} \right) \left( \sigma(g,h^{-1}v_1)\sigma(h,h^{-1}v_1)^{-1} \right) \\
&\subseteq \sigma(h,h^{-1}v_0)S_{(i',j')}\sigma(h,h^{-1}v_1)^{-1} \\
&=  \sigma(h,h^{-1}v_0)S_{(i',j')}\sigma(h,h^{-1}v_0)^{-1} = S_{(i,j)}.
\end{align*}
The group $G$ contains $F$, so $G$ is arc-transitive on $T$.  Moreover, there is $g \in G$ with $|g|=1$ such that $\sigma(g,v) = c$ for all $v \in VT$, where $c$ is the standard generator $1+n\Zb$ of $C$.  We choose the element $g$ to have an axis $\gamma$ passing through vertices $x_j := \gamma(j)$ for $j \in \Zb$, such that the edge $(x_j,x_{j+1})$ is labelled $(0,j)$.  Let $U = G_{x_1,x_{n}}$; we claim that $U$ is \emph{not} tidy for $g$.  Specifically, consider how $U$ acts on $Y_1 \times Y_n$, where $Y_1$ is the $U$-orbit of $(x_0,x_1)$ and $Y_n$ is the $U$-orbit of $(x_n,x_{n+1})$.  Taking labellings gives a bijection from $Y_1$ to $(\Zb/n\Zb \setminus \{0\}) \times \{0\}$ and also from $Y_n$ to $(\Zb/n\Zb \setminus \{0\}) \times \{0\}$, and $U$ acts on both $Y_1$ and $Y_n$ as the stabilizer of $(0,0)$ in $S_0$.  However, given $u \in U$, then $\sigma(u,x_1)$ and $\sigma(u,x_n)$ belong to the same left coset of the subgroup $\prod^{n-1}_{j=1}S_j$ of $A$, so $u$ induces the same permutation of the labels of $Y_1$ as it does on $Y_n$.  Hence $U$ does not act transitively on $Y_1 \times Y_n$, which also means that $U$ does not act transitively on $U\xi_-(g) \times U\xi_+(g)$, so $U$ fails condition $\GTA$.  Indeed, one can deduce that $G_{\gamma(i),\gamma(j)}$ fails $\GTA$, and hence fails to be tidy for $g$, whenever there exists $t \in \Zb$ such that $t+1 \le i < j \le t+n$.
\end{ex}

\section{The contraction group}\label{sec:contraction}

Recall that in a topological group $G$, the contraction group of $g \in G$ (modulo some closed $\grp{g}$-invariant subgroup $K$) consists of those elements $h \in G$ such that $g^nhg^{-n}K$ converges to the trivial coset in $G/K$ as $n \rightarrow +\infty$.  In this section we give a geometric interpretation for the contraction group, for a suitable actions of a \tdlc group $G$ on an NPC space $X$.

In this subsection we will be taking contraction groups modulo a compact $\grp{g}$-invariant subgroup $K$.  In \tdlc groups, this does not create any serious complications, thanks to Theorem~\ref{thm:bw:relative_contraction}.  However, it is useful to work modulo a compact $\grp{g}$-invariant subgroup for two reasons.  First, we want to consider contraction in terms of a closed convex $\grp{g}$-invariant (but not necessarily $G$-invariant) subspace $Y$, so we can only hope to describe the set $\con_G(g/K)$ where $K = \Fix_G(Y)$.  Second, in general contraction groups in \tdlc groups are not closed, even in the geometrically benign context of a proper continuous action on a locally finite tree; however, scale theory leads to some more natural closed groups of the form $\con_G(g/K)$, and we want to understand these geometrically.

Here is one interpretation of what it means for elements of a locally compact group to converge to the trivial isometry, relative to the fixator of some subspace $Y$.

\begin{lem}\label{lem:basic_contraction}
Let $X$ be a metric space, let $G$ be a locally compact group acting properly and continuously by isometries on $X$, and let $\emptyset \neq Y \subseteq X$ with $K:=\Fix_G(Y)$.  Let $(g_n)$ be a net of elements of $G$.  Then $g_nK \rightarrow K$ if and only if $d(g_ny,y) \rightarrow 0$ for every $y \in Y$.
\end{lem}

\begin{proof}
If $g_nK \rightarrow K$, then by continuity of the action we have $d(g_ny,y) \rightarrow 0$ for every $y \in Y$.  Conversely, suppose that $d(g_ny,y) \rightarrow 0$ for every $y \in Y$.  Since $G$ acts properly, we see that $(g_n)$ is eventually confined to a compact identity neighbourhood $L$, which we can take to satisfy $L = LK$.  Now let $O$ be an arbitrary identity neighbourhood.  Since $K$ is the fixator of $Y$, for each $l \in L \setminus OK$ there is some $y \in Y$ and $\delta > 0$ such that $d(ly,y) \ge 2\delta$, so by continuity, there is a neighbourhood $L'$ of $l$ such that $d(l'y,y) > \delta$ for all $l' \in L'$.  Given that $(L \smallsetminus OK)/K$ is compact, in fact there are points $y_1,\dots,y_k \in Y$ and distances $\delta_1,\dots,\delta_k > 0$ such that
\[
\forall g \in L: (\forall 1 \le i \le k: d(gy_i,y_i) \le \delta_i ) \Rightarrow g \in OK.
\]
In particular, by the above expression we deduce that $g_n \in OK$ eventually.  Since $O$ was arbitrary we conclude that $g_nK \rightarrow K$.
\end{proof}

We now specialize to the case of locally discrete actions.  In that case we can simplify the interpretation of Lemma~\ref{lem:basic_contraction} by observing that
\[
d(g_ny,y) \rightarrow 0 \quad \Leftrightarrow \quad g_ny = y \text{ eventually}.
\]
We define the pointwise limit inferior $\liminf (X_n)$ of a sequence (or net) of subsets of $X$ to consist of all points $x$ such that $x \in X_n$ eventually.

\begin{cor}\label{cor:basic_contraction}
Let $X$ be a metric space and let $G$ be a locally compact group acting locally discretely and properly by isometries on $X$.  Let $\emptyset \neq Y \subseteq X$ with $K:=\Fix_G(Y)$.  Let $u,g \in G$, such that $gKg\inv = K$.
\begin{enumerate}[(i)]
\item Suppose $u \in \con_G(g/K)$.  Then for all $y \in Y$, there is some $n_0$ such that $ug^{-n}y = g^{-n}y$ for all $n \ge n_0$.  In other words, for all $\epsilon >0$,
\[
Y \subseteq \liminf_{n \rightarrow \infty}g^nX^u.
\]
\item Suppose for every $y \in Y$ that $d(ug^{-n}y,g^{-n}y) \rightarrow 0$.  Then $u \in \con_G(g/K)$.
\end{enumerate}
\end{cor}

We now prove the characterization of compact subsets of contraction groups in terms of absorbing sets, as stated in Theorem~\ref{thm:intro_contraction}.

\begin{proof}[Proof of Theorem~\ref{thm:intro_contraction}]
First, consider the case that $g$ is bounded.  Since the action is proper, it follows that $\grp{g}$ is contained in a compact set, and indeed $(g^n)$ has a subsequence converging to the identity in $G$, from which it is clear that $\con_G(g/K) = K$.

From now on we may assume that $g$ is hyperbolic, with repelling point $\xi_-$.  Fix a base point $y_0 \in Y$, let $r$ be a positive real number and let $Z_r = \{y \in Y \mid d(y_0,y) \le r\}$.  Set $y_n = g^{-n}y_0$; fix also a ray $\rho$ representing $\xi_-$ with $\rho(0) = y_0$.  If $X$ is complete CAT(0) we can assume the image of $\rho$ is contained in an axis for $g$, so $y_n = \rho(n|g|)$; set $\kappa = |g|$.  If instead $X$ is a proper $\delta$-hyperbolic space, then since $\xi_-$ is the repelling point of $g$, by Lemma~\ref{lem:close_geodesic} there are $\kappa \in \Rb_{>0}$ and $t_n \rightarrow +\infty$ such that $d(y_n,\rho(t_n))+|t_{n+1}-t_n| \le \kappa$ for all $n$.

Since the action of $G$ is locally discrete, for each $y \in Z_r$ there is an open neighbourhood $O_y$ of $y$ fixed by an open subgroup of $G$.  Since $X$ is proper, finitely many such neighbourhoods suffice to cover $Z_r$, so there is a compact open subgroup $U_r$ of $G$ that fixes $Z_r$.  It follows that the fixator $H_r$ of $Z_r$ in $G$ contains $U_r$; in particular, $H_r$ is open.  Since the action is proper, each of the subgroups $H_r$ is also compact.

Suppose that the compact subset $C$ of $G$ is contained in $\con_G(g/K)$.  Then by Lemma~\ref{lem:uniform_contract}, for each $r > 0$ there is $n_r \in \Nb$ such that $g^nCg^{-n}$ is a subset of $H_{r+\kappa}$ for all $n \ge n_r$.  Thus $\grp{C}$ fixes $g^{-n}Z_{r+\kappa}$ for all $n \ge n_r$; note that $g^{-n}Z_{r+\kappa} = \{y \in Y \mid d(y_n,y) \le r+\kappa\}$, which contains the ball of radius $r$ around $\rho(t)$ for all $t_n \le t \le t_{n+1}$.  Thus $X^{\grp{C}}$ is an absorbing set for $\xi_-$ within $Y$.

Conversely, suppose $X^{\grp{C}}$ is an absorbing set for $\xi_-$ within $Y$ and let $h \in C$.  Then $h$ fixes the set $\{y \in Y \mid d(y_n,y) \le r_n\}$, where $r_n$ is some sequence tending to $+\infty$; by replacing $r_n$ with $\inf\{r_{n'} \mid n' \ge n\}$, we may assume $r_n \le r_{n+1}$ for all $n \in \Nb$.  Hence for all $n' \ge n$, we see that $g^{n'}hg^{-n'}$ fixes the set 
\[
\{y \in Y \mid d(g^{n'}y_{n'},y) \le r_{n'}\} = Z_{r_{n'}},
\]
so $g^{n'}hg^{-n'} \le H_{r_n}$.  Since $r_n \rightarrow +\infty$, we see that $K$ is the intersection of the descending sequence of subgroups $(H_{r_n})$; since each of the subgroups $H_{r_n}$ is compact and open in $G$, in fact the sets $H_{r_n}/K$ form a base of neighbourhoods of the trivial coset in $G/K$.  Thus $g^{n}hg^{-n}K \rightarrow K$ in $G/K$ as $n \rightarrow +\infty$, so $h \in \con_G(g/K)$.
\end{proof}

\begin{proof}[Proof of Corollary~\ref{cor:intro_contraction_closed}]
The equivalence of (a) and (b) is clear from Corollary~\ref{cor:contraction_closure}.

Let $W$ be a compact open subgroup of $G$, for example $W = G_x$ for some $x \in X$.  We see that $W = V \cap \con_G(g/K)$ is an open subgroup of $\con_G(g/K)$.  Suppose $\con_G(g/K)$ is closed.  Then $W$ is compact, so by Lemma~\ref{lem:uniform_contract}, for every compact open subgroup $U$ of $G$, there exists $n_U \in \Nb$ such that $g^nWg^{-n} \le U$ for all $n \ge n_U$.  By Theorem~\ref{thm:intro_contraction}, we see that $X^W$ is absorbing for $\xi_-(g)$ within $Y$.  Conversely, suppose that $X^W$ is an absorbing set for $\xi_-(g)$ within $Y$.  Since $X^W = X^{\ol{W}}$, it follows from Theorem~\ref{thm:intro_contraction} that $\ol{W} \le \con_G(g/K)$, and since $V$ is a clopen identity neighbourhood in $G$, we deduce that $\con_G(g/K)$ is closed.  Thus (a) and (c) are equivalent.
\end{proof}

\begin{proof}[Proof of Corollary~\ref{cor:intro_contraction_tidy}]
Let $C = \con_G(g/K)$.  Taking $V$ minimizing for $g$ and $W = V \cap C$, we then have
\[
s_G(g^{-1}) = |g^{-1}Vg:g^{-1}Vg \cap V| \ge |g^{-1}Wg:g^{-1}Wg \cap W| \ge s_C(g^{-1}).
\]
By Lemma~\ref{lem:scale_subgroup} we have $s_C(g^{-1}) = s_G(g^{-1})$, so in fact $|g^{-1}Wg:g^{-1}Wg \cap W| = s_C(g^{-1})$, ensuring that $W$ is minimizing for $g^{-1}$, and hence also for $g$.  If $C$ is compact then $s_G(g^{-1})=1$ by Lemma~\ref{lem:BW_scale_one}, so $C=W$, and then we see that the only way $g$ can be contracting on $C$ modulo $K$ is if $C=K$.  From now on we may assume that $C$ is not compact (equivalently, $s_G(g^{-1})>1$).

The set of fixed points of any subgroup is closed and convex, so $Z = X^W \cap Y$ is closed and convex.

Clearly $C \le \para_G(g)$, so by Lemmas~\ref{lem:BW_scale_one} and~\ref{lem:scale_subgroup} we have $s_C(g) =1$; thus $gWg^{-1} < W$, and hence
\[
X^W \subseteq X^{gWg^{-1}} = gX^W.
\]
Since $Y$ is $\grp{g}$-invariant, we deduce that $Z \subseteq gZ$, so $(g^nZ)_{n \in \Zb}$ consists of subsets of $Y$ that are totally ordered by inclusion.  The fact that $Z$ is absorbing for $\xi_-(g)$ within $Y$ ensures that in fact $\bigcup_{n \in \Zb} g^nZ = Y$.  On the other hand, since $\bigcup_{n \in \Zb}g^nWg^{-n} = C$, we see that the set $Z_0 := \bigcap_{n \in \Zb} g^nZ$ can be written as $Z_0 = X^C \cap Y$.  Since $C$ is not compact and $G$ acts properly, it follows that $Z_0$ is empty.
\end{proof}

We finish with a well-known example of a \tdlc group acting on a tree, to illustrate Corollary~\ref{cor:intro_contraction_tidy} in the case where $\con_G(g)$ itself is closed.

\begin{ex}\label{ex:SL2}
Let $\widetilde{G} = \mathrm{GL}_2(\Qb_p)$ and let $X$ be the Bruhat--Tits tree of $\widetilde{G}$; for more details of the construction, see \cite[Chapter~II \S1]{Serre:trees}.  Let $S$ be the group of scalar matrices in $\widetilde{G}$.  (The group acting on the tree that we are interested in is actually $G = \widetilde{G}/S$, but it is easier to consider the general linear group in order to use its action on $\Qb^2_p$.)  In particular, the vertices of $X$ are the $S$-orbits $[L]$ of lattices $L$ of $\Qb^2_p$, where \defbold{lattice} means a free $\Zb_p$-submodule of rank $2$.  The distance from $[L]$ to $[L']$ is given as $|i-j|$ where $L = p^i\Zb_pe'_1 + p^j\Zb_pe'_2$ for some basis $\{e'_1,e'_2\}$ of $L'$, and the number $|i-j|$ is unaffected by multiplying $L$ by a scalar.  In particular, if we take representatives such that $L \le L'$ but $L \nleq pL'$ (which is always possible), then $L/L'$ is trivial or cyclic and we see that $d([L],[L']) = \log_p(|L:L'|)$.  Fix a basis $\{e_1,e_2\}$ for $\Qb^2_p$ and $L_0 := \Zb_pe_1 + \Zb_pe_2$, and for $l \in \Qb^*_p$, write $\nu(l)$ for the largest integer $n$ such that $p^{-n}l \in L_0$.  Let $\mc{L}$ be the set of open subgroups $L$ of $L_0$ such that $L \nleq pL_0$; then $\mc{L}$ has a unique representative of each vertex.
  
We consider a standard hyperbolic element $g$, its contraction group $C_g :=\con_{\widetilde{G}}(g/S)$ (which in this case is closed in $\widetilde{G}$), and $W = C_g \cap \widetilde{G}_{[L_0]}$:
\[
g = 
\begin{pmatrix}
p & 0 \\
0 & 1
\end{pmatrix}; \quad
C_g = \{u_a \mid a \in \Qb_p\},\; \text{where} \quad
u_a = 
\begin{pmatrix}
1 & a \\
0 & 1
\end{pmatrix}; \quad
W = \{u_a \mid a \in \Zb_p\}.
\]
Let $L_i$ be $e_2+p^iL_0$ if $i > 0$ and $e_1+p^{-i}L_0$ if $i<0$; then the convex hull of $\{[L_i] \mid i \in \Zb\}$ is the axis of $g$, with $g[L_i] = [L_{i+1}]$ for all $i \in \Zb$.  We recover the set $Z$ as in Corollary~\ref{cor:intro_contraction_tidy} as $X^W$; we claim that in fact $X^W = Z_0$, where $Z_0$ is the convex hull of the horosphere centred at $\xi_-(g)$ that passes through $[L_0]$.  Given $n \ge 0$ and $L \in \mc{L}$ such that $d_X([L_0],[L]) \le n$, observe that $L$ contains $p^nL_0$, and hence for all $a \in \Zb_p$, the vertex $[L]$ is fixed by $g^{n}u_ag^{-n} = u_{p^na}$.  Thus $W$ fixes pointwise the ball of radius $n$ around $[L_{-n}]$ for all $n \ge 0$; hence $W$ fixes $Z_0$.

Conversely, suppose the vertex $[L]$ is fixed by $u_1 \in W$; we can take a representative $L = l + p^kL_0$ where $l = l_1e_1 + l_2e_2$ for $l_1,l_2 \in \Zb_p$, such that $\min\{\nu(l_1),\nu(l_2)\}=0$ and $k \ge 0$.  We can ignore the case $l \in e_1+p^kL_0$, as we know the vertex $[L_{-k}] = [e_1+p^kL_0]$ is fixed by $W$, so assume $\nu(l_2)<k$.  Given $a \in \Zb_p$, we see that $u_1L \in \mc{L}$, so 
\begin{align*}
u_{1}[L] = [L] \Rightarrow u_{1}L = L &\Rightarrow \exists \lambda \in \Zb^*_p: u_{1}l-\lambda l \in p^kV_0\\
 &\Rightarrow \exists \lambda \in \Zb^*_p: ((1-\lambda)l_1 + l_2)e_1 + (1-\lambda)l_2e_2 \in p^kV_0\\
 &\Rightarrow \exists \lambda' \in \Zb_p: \nu(\lambda'l_1 + l_2) \ge k, \; \nu(\lambda'l_2) \ge k.
\end{align*}
To satisfy the last line, we need $\nu(l_2) = \nu(\lambda'l_1)$ and $\nu(\lambda') \ge k-\nu(l_2)$, which together imply $2\nu(l_2) \ge \nu(l_1) + k$.  In particular, $\nu(l_2) > 0$, so $\nu(l_1)=0$.  Let $r = k-\nu(l_2)$.  We see that $p^rL = p^rl_1e_1 + p^{k+r}L_0 = p^rL_{-k}$, so 
\[
d([L],[L_{-k}]) \le \log_p(|L_{-k}:p^rL_{-k}|) = 2r \le k,
\]
and hence $[L] \in Z_0$.

In particular, $X^W$ is clearly an absorbing set for $\xi_-(g)$ in $X$, and we have $\bigcup_{n \in \Zb}g^n X^W = X$.
\end{ex}

\end{document}